\documentclass[11pt]{amsart}
\usepackage{lmodern}
\usepackage[T1]{fontenc}
\usepackage{microtype}
\usepackage{amssymb}
\usepackage{amsthm}
\usepackage{amscd}
\usepackage{hyperref}
\usepackage{cite}
\usepackage{color}

\def\dcl{\mathop{\text{dcl}_{}}\nolimits}					
\def\acl{\mathop{\text{acl}_{}}\nolimits}					
\def\lN{\mathop{\mathcal{L}_{\mathbb{N}}}\nolimits}	 		
\def\ln{\mathop{\mathcal{L}_{[n]}}\nolimits}					
\def\equalinlaw{\mathop{=_{\mathcal{D}}}\nolimits}			
\def\age{\mathop{\text{age}}\nolimits}					

\def\Nb{\mathop{\mathbb{N}_{}}\nolimits}					
\def\im{\mathop{\text{im}}\nolimits}						
\def\ar{\mathop{\mathrm{ar}}\nolimits}	

\newtheorem{theorem}{Theorem}[section]
\newtheorem{lemma}[theorem]{Lemma}
\newtheorem{prop}[theorem]{Proposition}

\newtheorem{question}[theorem]{Question}

\theoremstyle{definition}
\newtheorem{definition}[theorem]{Definition}

\newtheorem{remark}[theorem]{Remark}

\newtheorem{example}[theorem]{Example}

\newtheorem{claim}{Claim}

 \newenvironment{claimproof}{\begin{proof}}{\end{proof}}

\def \rng{\operatorname{rng}}

\makeatletter

\def\dotminussym#1#2{%
  \setbox0=\hbox{$\m@th#1-$}%
  \kern.5\wd0%
  \hbox to 0pt{\hss\hbox{$\m@th#1-$}\hss}%
  \raise.6\ht0\hbox to 0pt{\hss$\m@th#1.$\hss}%
  \kern.5\wd0}

\mathchardef\mhyphen="2D

\allowdisplaybreaks[2]

\begin{document}

\title{Relatively exchangeable structures}
\author{Harry Crane and Henry Towsner}
\address {Department of Statistics \& Biostatistics, Rutgers University, 110 Frelinghuysen Avenue, Piscataway, NJ 08854, USA}
\email{hcrane@stat.rutgers.edu}
\urladdr{\url{http://stat.rutgers.edu/home/hcrane}}
\address {Department of Mathematics, University of Pennsylvania, 209 South 33rd Street, Philadelphia, PA 19104-6395, USA}
\email{htowsner@math.upenn.edu}
\urladdr{\url{http://www.math.upenn.edu/~htowsner}}
\thanks{H.\ Crane is partially supported by NSF grant DMS-1308899 and NSA grant H98230-13-1-0299.}
\thanks{ H.\ Towsner is partially supported by NSF grant DMS-1340666.}


\date{\today}

\begin{abstract}
We study random relational structures that are \emph{relatively exchangeable}---that is, whose distributions are invariant under the automorphisms of a reference structure $\mathfrak{M}$.
When $\mathfrak{M}$ is {\em ultrahomogeneous} and has {\em trivial definable closure}, all random structures relatively exchangeable with respect to $\mathfrak{M}$ satisfy a general Aldous--Hoover-type representation.
If $\mathfrak{M}$ also satisfies the {\em $n$-disjoint amalgamation property} ($n$-DAP) for all $n\geq1$, then relatively exchangeable structures have a more precise description whereby each component depends locally on $\mathfrak{M}$.
\end{abstract}

\maketitle

\section{Introduction}\label{section:introduction}

\subsection{Relational structures}

A \emph{signature} is a finite\footnote{All structures in this paper are finite relational structures.} set $\mathcal{L}=\{R_1,\ldots,R_r\}$ and, for each $j\leq r$, a positive integer $\ar(R_j)$, called the \emph{arity} of $R_j$.  An {\em $\mathcal{L}$-structure} is a collection $\mathfrak{M}=(M,\mathcal{R}_1,\ldots,\mathcal{R}_r)$, where $M$ is a set and $\mathcal{R}_j\subseteq M^{\ar(R_j)}$ for each $j\in[1,r]:=\{1,\ldots,r\}$.  
We write $|\mathfrak{M}|:=M$ and $R_j^{\mathfrak{M}}:=\mathcal{R}_j$ for each $j\in[1,r]$.
In general, we write $\mathcal{L}_M$ to denote the set of $\mathcal{L}$-structures $\mathfrak{M}$ for which $|\mathfrak{M}|=M$.
Specifically, $\lN$ denotes $\mathcal{L}$-structures with $|\mathfrak{M}|=\mathbb{N}:=\{1,2,\ldots\}$ and $\ln$ denotes $\mathcal{L}$-structures with $|\mathfrak{M}|=[n]:=[1,n]$.

Every injection $\phi:M'\rightarrow M$ maps $\mathcal{L}_M$ into $\mathcal{L}_{M'}$ in the usual way: $\mathfrak{M}\mapsto\mathfrak{M}^{\phi}:=(M',\mathcal{R}_1^{\phi},\ldots,\mathcal{R}_r^{\phi})$ with
\[(s_1,\ldots,s_{\ar(R_j)})\in\mathcal{R}_j^{\phi}\quad\Longleftrightarrow\quad(\phi(s_1),\ldots,\phi(s_{\ar(R_j)}))\in\mathcal{R}_j.\]
We call $\phi$ an \emph{embedding} of $\mathfrak{M}^{\phi}$ into $\mathfrak{M}$, written $\phi:\mathfrak{M}^{\phi}\rightarrow\mathfrak{M}$.  In particular, every permutation $\sigma:M\rightarrow M$ determines a {\em relabeling} of any $\mathfrak{M}\in\mathcal{L}_{M}$.
When $M'\subset M$, the inclusion map, $s\mapsto s$, determines the {\em restriction} of $\mathfrak{M}$ by
\[\mathfrak{M}|_{M'}:=(M',\mathcal{R}_1\cap M'^{\ar(R_1)},\ldots,\mathcal{R}_r\cap M'^{\ar(R_r)}).\]

If $\mu$ is a probability measure on $\mathcal{L}_M$, we write $\mathfrak{X}\sim\mu$ to denote that $\mathfrak{X}$ is a random structure chosen according to $\mu$; in this case we call $\mathfrak{X}$ a {\em random $\mathcal{L}$-structure on $M$}. 
We call a pair $\mathfrak{X}$ and $\mathfrak{Y}$ of random $\mathcal{L}$-structures on $M$ \emph{equal in distribution}, written $\mathfrak{X}\equalinlaw\mathfrak{Y}$, if $\mathbb{P}(\mathfrak{X}|_S=\mathfrak{S})=\mathbb{P}(\mathfrak{Y}|_S=\mathfrak{S})$  for every $\mathfrak{S}\in\mathcal{L}_S$, for all finite $S\subseteq M$.

\subsection{Relative exchangeability}

Special cases of $\mathcal{L}$-structures include binary relations, set partitions, undirected graphs, triangle-free graphs, as well as composite objects, e.g., a set together with a binary relation, a pair of graphs, etc.
We are particularly interested in random $\mathcal{L}$-structures that satisfy natural invariance properties with respect to the symmetries of another structure, of which exchangeability is a special case.

\begin{definition}[Exchangeability]\label{defn:exchangeability}
Let $\mathcal{L}$ be a signature.
A random $\mathcal{L}$-structure $\mathfrak{X}$ is {\em exchangeable} if $\mathfrak{X}^{\sigma}\equalinlaw\mathfrak{X}$ for all permutations $\sigma:|\mathfrak{X}|\rightarrow|\mathfrak{X}|$.
We also call a probability measure $\mu$ {\em exchangeable} whenever $\mathfrak{X}\sim\mu$ is an exchangeable $\mathcal{L}$-structure.
\end{definition}

Given a large structure $\mathfrak{U}=(\Omega,\mathcal{R}_1,\ldots,\mathcal{R}_r)$ and a probability measure $\mu$ on $\Omega$, we can obtain an exchangeable random $\mathcal{L}$-structure $\mathfrak{X}=(\Nb,\mathcal{X}_1,\ldots,\mathcal{X}_r)$ by sampling elements $\phi(1),\phi(2),\ldots$ independently and identically distributed (i.i.d.)~from $\mu$ and then defining $\mathfrak{X}=\mathfrak{U}^{\phi}$.
Explicit representations of exchangeable structures are detailed in the work of de Finetti \cite{deFinetti1937}, Aldous \cite{Aldous1981}, Hoover \cite{Hoover1979}, and Kallenberg \cite{KallenbergSymmetries}.
 As a special case, the Aldous--Hoover theorem \cite{Aldous1981,Hoover1979} characterizes the exchangeable random $k$-ary hypergraphs $\mathfrak{X}=(\Nb,\mathcal{X})$---that is, the exchangeable random structures with a single symmetric $k$-ary relation---through the decomposition
\begin{equation}\label{eq:AH-rep}
\vec x\in \mathcal{X}\quad\Longleftrightarrow\quad f((\xi_{s})_{s\subseteq \rng\vec x})=1,\end{equation}
where $f$ is Borel measurable, the random variables $\xi_{s}$ are i.i.d.\ Uniform$[0,1]$, and $\rng\vec x$ is the set of distinct elements in $\vec x$.  
For instance, an exchangeable random graph can be generated by specifying a function $f:[0,1]^4\rightarrow\{0,1\}$ with $f(\cdot,b,c,\cdot)=f(\cdot,c,b,\cdot)$, selecting independent Uniform$[0,1]$ parameters $\xi_{\emptyset}$, $\xi_{\{i\}}$ for each $i\in \Nb$, and $\xi_{\{i,j\}}$ for each pair $i<j$, and including the edge $\{i,j\}$ exactly when
\[f(\xi_{\emptyset},\xi_{\{i\}},\xi_{\{j\}},\xi_{\{i,j\}})=1.\]

Exchangeable structures not only play a fundamental role in probability theory \cite{AldousExchangeability,KallenbergSymmetries}, Bayesian inference \cite{deFinetti1937}, and applications in population genetics \cite{Kingman1978a} but also have a natural place in the study of homogeneous structures in combinatorics \cite{PetrovVershik2010} and mathematical logic \cite{AFP,AFKP,AFNP}.
In many applications, e.g., spin-glass models in statistical physics \cite{AustinPanchenko2014} and combinatorial stochastic processes \cite{Bertoin2006,Pitman2005}, a random structure $\mathfrak{X}$ is only invariant under relabeling by permutations that fix certain substructures of a reference object $\mathfrak{M}$, leading to our notion of relative exchangeability.

\begin{definition}[Relative exchangeability]\label{defn:relative exchangeability}
Let $\mathcal{L},\mathcal{L}'$ be signatures and $\mathfrak{M}$ be an $\mathcal{L}$-structure.
A random $\mathcal{L}'$-structure $\mathfrak{X}$ is called {\em relatively exchangeable with respect to $\mathfrak{M}$}, alternatively {\em exchangeable relative to $\mathfrak{M}$} or {\em $\mathfrak{M}$-exchangeable}, if $\mathfrak{X}|_T^{\phi}\equalinlaw\mathfrak{X}|_S$ for all embeddings $\phi:\mathfrak{M}|_S\rightarrow\mathfrak{M}|_T$ with $S,T\subseteq|\mathfrak{M}|$ finite.
\end{definition}

\begin{remark}\label{rmk:distinction}
Relative exchangeability requires more than invariance of $\mathfrak{X}$ with respect to the automorphisms of $\mathfrak{M}$; it requires that when $S$ and $T$ are isomorphic substructures, the marginal distributions of $\mathfrak{X}|_S$  and $\mathfrak{X}|_T$ are the same.  
This means that the marginal distribution of substructures $\mathfrak{X}|_T$ depends only on the symmetries of the associated substructure $\mathfrak{M}|_T$.
In particular, $\phi:\Nb\rightarrow\Nb$ may be an injection that is not an automorphism of $\mathfrak{M}$ but whose domain restriction $\phi\upharpoonright T: T\to T'$ is an embedding $\mathfrak{M}|_T\to\mathfrak{M}$.
In this case, $\mathfrak{M}^{\phi}|_{T}=\mathfrak{M}|_{T'}^{\phi\upharpoonright T}=\mathfrak{M}|_{T}$ and $\mathfrak{X}^{\phi}|_{T}=\mathfrak{X}|_{T'}^{\phi\upharpoonright T}\equalinlaw\mathfrak{X}|_{T}$.
\end{remark}

\begin{remark}\label{rmk:moot}
The distinction noted in Remark \ref{rmk:distinction} does not factor into our main discussion.
Since we always assume $\mathfrak{M}$ is {\em ultrahomogeneous} (see coming paragraphs and Definition \ref{def:ultra}), Definition \ref{defn:relative exchangeability} is equivalent to the apparently weaker requirement that $\mathfrak{X}$ is invariant with respect to the automorphisms of $\mathfrak{M}$.
In general, the stronger condition of Definition \ref{defn:relative exchangeability} is more relevant in many statistical and probabilistic contexts, and its equivalence to the weaker form is useful in our proofs.
\end{remark}
\begin{remark}
The classical definition of exchangeability corresponds to relative exchangeability with $\mathcal{L}=\emptyset$.
\end{remark}

The notion of relative exchangeability is most interesting when $\mathfrak{M}$ has many partial automorphisms.  One natural condition to place on $\mathfrak{M}$ is \emph{trivial (group-theoretic) definable closure} (Definition \ref{defn:TDC}), which says that, for any finite subset $s$ and any $a\not\in s$, there are automorphisms $\phi$ of $\mathfrak{M}$ that fix every element of $s$ but for which $\phi(a)\neq a$.  Ackerman, Freer \& Patel \cite{AFP} have previously shown that structures $\mathfrak{M}$ with trivial definable closure are exactly those for which there is an exchangeable probability measure concentrated on the class of structures isomorphic to $\mathfrak{M}$.

In this paper we will consider the case where $\mathfrak{M}$ has an additional property, ultrahomogeneity (Definition \ref{def:ultra}), which says that every finite partial automorphism extends to a full automorphism.  Ultrahomogeneity is a natural assumption in our intended application of relative exchangeability to Markov processes in spaces of countable relational structures \cite{CraneTowsner2015b}.  
In that setting, we often deal with ensembles of structures, not all of which are ultrahomogeneous but so that the ensemble embeds into a common ultrahomogeneous structure in a suitable way.
Together, these properties imply a strong representation for $\mathfrak{X}$, with each piece of $\mathfrak{X}$ depending locally on $\mathfrak{M}$, as we make precise in Theorem \ref{thm:nice1}.   

\subsection{Main theorems}
Above all we seek analogs of the Aldous--Hoover theorem for relatively exchangeable structures.
A formal description of our main theorems requires several technical conditions, which we defer until later.
For now we settle for an overview.

Our most general result gives a representation for all random structures that are exchangeable relative to $\mathfrak{M}$ with trivial definable closure and ultrahomogeneity.
We use the main theorem in \cite{AFP} to prove the generic Aldous--Hoover-type representation for any $\mathfrak{M}$-exchangeable structure when $\mathfrak{M}$ is ultrahomogeneous (Definition \ref{def:ultra}).
We can, however, refine the generic Aldous--Hoover representation when $\mathfrak{M}$ satisfies stronger properties.

In the more general setting of Theorem \ref{thm:main1}, we show that each $\mathfrak{X}|_{\rng\vec x}$ depends on the entire initial substructure $\mathfrak{M}|_{[1,\max\vec x]}$.  
Under the additional assumption that $\mathfrak{M}$ satisfies the $n$-disjoint amalgamation property for all $n\geq1$ (Definition \ref{def:DAP}),  Theorem \ref{thm:nice1} gives a stronger representation which describes $\mathfrak{X}|_{\rng\vec x}$ in terms of $\mathfrak{M}|_{\rng\vec x}$ and random variables similar to those in the usual Aldous--Hoover Theorem\footnote{A similar result has been independently shown by Ackerman \cite{2015arXiv150906170A} using a different argument which applies to a slightly different class of structures.}.  The $n$-disjoint amalgamation property is a finite amalgamation property that ensures any consistent collection of substructures can be embedded into some other substructure of $\mathfrak{M}$.

In the classical theory of exchangeability \cite{Aldous1981,deFinetti1937,Hoover1979}, dissociated structures play a key role as extreme points in the space of exchangeable random objects.
\begin{definition}[Dissociated random structures]\label{defn:dissociated}
A random $\mathcal{L}$-structure $\mathfrak{X}$ is {\em dissociated} if $\mathfrak{X}|_S$ and $\mathfrak{X}|_T$ are independent for all disjoint subsets $S,T\subseteq|\mathfrak{X}|$.
\end{definition}

In particular, exchangeable processes can be decomposed into an average over a family of dissociated processes.  We show that a similar decomposition exists for $\mathfrak{M}$-exchangeable processes when $\mathfrak{M}$ is sufficiently nice.

\subsection{Connections to the literature}
Our main theorems extend representations of exchangeable structures to the more general setting of relatively exchangeable structures.
Relatively exchangeable structures naturally appear when one considers dependent sequences of exchangeable structures, as in the study of combinatorial Markov processes \cite{Pitman2005}.    
Such processes serve as models in a wide range of applications, some mentioned above, and so are of interest on their own.
These considerations invoke certain other technicalities from stochastic process theory, which we leave to the more probability-focused companion paper \cite{CraneTowsner2015b}.

Prior work of Diaconis \& Janson \cite{MR2463439} highlights the connection (via the Aldous--Hoover theorem) between exchangeable random graphs and the Lov\'asz--Szegedy theory of graph limits \cite{LovaszSzegedy2006}.
The extension of Aldous--Hoover to exchangeable $\mathcal{L}$-structures (Theorem \ref{AH-general}) makes plain the analogous connection between exchangeable $\mathcal{L}$-structures and the generalization of graph limits to $\mathcal{L}$-structures \cite{AroskarCummings2014}.

\subsection{Notation}

We adopt the following notational conventions: $\mathcal{L}$ and $\mathcal{L}'$ always denote signatures.  
 In general, we use fraktur letters ($\mathfrak{M}$, $\mathfrak{N}$, $\mathfrak{S}$, $\mathfrak{T}$) to denote structures.  
The base set is indicated by plain Roman letters ($M$, $N$, $S$, $T$).

For $\vec x=(x_1,\ldots,x_k)$, we write $\rng\vec x=\{x_1,\ldots,x_k\}$ to denote the set of distinct elements in $\vec x$ and we write $\vec y\subseteq\vec x$ to denote that $\vec y$ occurs as a subsequence of $\vec x$, that is, $\vec y=(x_{p_1},\ldots,x_{p_m})$ for an increasing sequence $p_1<\cdots<p_m$.

\subsection{Outline}

We organize the rest of the paper as follows.  In Section \ref{section:general} we describe the known results about exchangeable structures and Fra\"isse classes, which we need in the rest of the paper; we also provide context for the results that follow.  In Section \ref{section:summary} we describe two of our main results, a weaker representation whenever $\mathfrak{M}$ is ultrahomogeneous, and a stronger one when $\mathfrak{M}$ also has $n$-DAP for all $n$.  In Section \ref{section:proofs} we prove these results, giving an Aldous--Hoover type theorem for relatively exchangeable structures.
Along the way we provide several illustrative examples that should build intuition for how the assumed properties of $\mathfrak{M}$ play a role in our representation.
\section{Exchangeable Structures}\label{section:general}

Before proving our main theorems about random $\mathcal{L}$-structures, we first review some known results.
Throughout the paper, we equip $\mathcal{L}_{\Nb}$ with the product-discrete topology induced by the ultrametric
\[d(\mathfrak{M},\mathfrak{M}'):=1/\max\{n\in\Nb:\,\mathfrak{M}|_{[n]}\neq\mathfrak{M}'|_{[n]}\},\quad\mathfrak{M},\mathfrak{M}'\in\mathcal{L}_{\mathbb{N}},\]
under which $\lN$ is compact.
Equipping $\mathcal{L}_{\Nb}$ with the corresponding Borel $\sigma$-field allows us to ignore measure-theoretic technicalities to every extent possible.

\subsection{The Aldous--Hoover Theorem}

The Aldous--Hoover theorem has been generalized to exchangeable structures other than hypergraphs, e.g., \cite{AroskarCummings2014,MR2463439,KallenbergSymmetries}, including asymmetric or reflexive relations.  
These considerations introduce some (mostly notational) complications to the representation in \eqref{eq:AH-rep}.  One approach, taken in \cite{AroskarCummings2014,MR2463439}, is to break a single relation into several correlated relations.
For instance, a binary relation $\mathcal{R}$ consists of a unary relation $\{x\mid (x,x)\in\mathcal{R}\}$ and four binary relations corresponding to the four possible cases for a pair $(x,y)$.  Here we adopt a more uniform approach by including a random \emph{ordering} in addition to Uniform$[0,1]$ random variables.

\begin{definition}
  When $s$ is a finite set, by a \emph{uniform random ordering of $s$}, we mean an ordering $\prec_{s}$ of $s$ chosen uniformly at random.  Given $\prec_{\rng\vec x}$, we write $\prec_{\vec x}$ for the ordering of $[1,|\rng\vec x|]$ induced by $i\prec_{\vec x}j$ if and only if $x_i\prec_{\rng\vec x}x_j$.  If $x_i=x_j$, then $i\not\prec_{\vec x}j$ and $j\not\prec_{\vec x}i$.
\end{definition}

\begin{remark}
Note that $\prec_s$ is an ordering of the \emph{set} $s$, and whenever $\vec x$ is a \emph{sequence} of elements from $s$, we use $\prec_s$ to induce an ordering $\prec_{\vec x}$ of $[1,|\rng\vec x|]$.  The important feature is that when $\vec x$ and $\vec y$ are different orderings of $s$, $\prec_{\vec x}$ and $\prec_{\vec y}$ are distinct but related orderings.

For example, for $\vec x=(x_1,\ldots,x_k)$ and $\sigma:[k]\rightarrow[k]$ a permutation, we observe that
\[i\prec_{\vec x} j\quad\Longleftrightarrow\quad\sigma(i)\prec_{\sigma\vec x}\sigma(j),\]
where $\sigma\vec x:=(x_{\sigma(1)},\ldots,x_{\sigma(k)})$.
In particular, $\prec_{(x,y)}$ is always the opposite of $\prec_{(y,x)}$.
\end{remark}

\begin{definition}
  Let $\mathcal{L}=\{R_1,\ldots,R_r\}$ be a language so that each $R_i$ has $ar(R_i)\leq k$ and let $f_1,\ldots,f_r$ be Borel functions.  The \emph{exchangeable structure generated by} $f_1,\ldots,f_r$ is the structure $\mathfrak{X}^*=(\Nb,R^{\mathfrak{X}^*}_1,\ldots,R^{\mathfrak{X}^*}_r)$ given by choosing $(\xi_s)_{s\subseteq\Nb:\,|s|\leq k}$ i.i.d.\ Uniform$[0,1]$ and $(\prec_s)_{s\subseteq\Nb:\,|s|\leq k}$ independent uniform random orderings and putting
\[\vec x\in R^{\mathfrak{X}^*}_i\quad\Longleftrightarrow\quad f_i((\xi_{s})_{s\subseteq\rng\vec x},(\prec_{\vec y})_{\vec y\subseteq\vec x})=1.\]
\end{definition}
\begin{remark}We usually omit $\prec_{\vec y}$ when $|\rng\vec y|\leq 1$ because such an ordering is trivial.\end{remark}

\begin{remark}\label{remark:order}
For definiteness, we assume arguments are listed in some fixed order---say, lexicographical order of subsequences of $\vec x$.  For instance, when we write $f_i((\xi_s)_{s\subseteq\{x,y\}},(\prec_{\vec y})_{\vec y\subseteq(x,y)})$, we mean
\[f_i(\xi_\emptyset,\xi_{\{x\}},\xi_{\{y\}},\xi_{\{x,y\}},\prec_{(x,y)}).\]
Similarly, when we write $f_i((\xi_s)_{s\subseteq\{x,y,z\}},(\prec_{\vec y})_{\vec y\subseteq(x,y,z)})$, we mean
\[f_i(\xi_\emptyset,\xi_{\{x\}},\xi_{\{y\}},\xi_{\{z\}},\xi_{\{x,y\}},\xi_{\{x,z\}},\xi_{\{y,z\}},\xi_{\{x,y,z\}},\prec_{(x,y)},\prec_{(x,z)},\prec_{(y,z)},\prec_{(x,y,z)}).\]
\end{remark}

\begin{remark}
The presence of the $\prec_{\vec y}$ allows us to communicate between different orderings of $s$ without giving precedence to some extrinsic ordering (like the ordering $<$ on $\mathbb{N}$).
For instance, if $R$ is a binary relation, then
\[(1,2)\in R^{\mathfrak{X}^*}\quad\Longleftrightarrow\quad f(\xi_\emptyset,\xi_{\{1\}},\xi_{\{2\}},\xi_{\{1,2\}},\prec_{(1,2)})=1\]
while
\[(2,1)\in R^{\mathfrak{X}^*}\quad\Longleftrightarrow\quad f(\xi_\emptyset,\xi_{\{1\}},\xi_{\{2\}},\xi_{\{1,2\}},\prec_{(2,1)})=1.\]
Since $\prec_{(1,2)}$ is always the opposite of $\prec_{(2,1)}$, these may have different values.  
Similarly,
\[(1,1)\in R^{\mathfrak{X}^*}\quad\Longleftrightarrow\quad f(\xi_{\emptyset},\xi_{\{1\}},\prec_{(1,1)})=1,\]
where $\prec_{(1,1)}$ is necessarily an empty ordering.

Note that when $s\subsetneq s'$, $\prec_s$ and $\prec_{s'}$ are not correlated.
In particular, $\prec_{s'}$ need not extend $\prec_s$.
\end{remark}

\begin{remark}
  This representation is somewhat redundant: instead of encoding the ordering of $x$ and $y$ into the random ordering $\prec_{\{x,y\}}$, one could use $\xi_{\{x\}}$ and $\xi_{\{y\}}$---say, by determining that $x\prec_{\{x,y\}}y$ if and only if $\xi_{\{x\}}<\xi_{\{y\}}$.  Without loss of generality, we could assume that $f_i$ is required to be symmetric in $(\xi_s)_{s\subseteq\rng\vec x}$, which would eliminate this redundancy.  Since this would add further complications to an already involved definition for no clear benefit, we do not do so.

On the other hand, we could drop the parameters $(\prec_{\vec y})_{\vec y\subseteq\vec x}$ entirely.  We do not do so, because this would violate the stratification of data provided by the representation.  Later we need to separate the {\em unary data}, such as $\xi_{\{x\}}$, from the {\em binary data}, such as $\xi_{\{x,y\}}$ and $\prec_{\{x,y\}}$, and we need the asymmetry given by $\prec_{\{x,y\}}$ to be part of the binary data.
\end{remark}

We can now state the Aldous--Hoover theorem in a useful general way.
\begin{theorem}[Aldous \cite{Aldous1981}, Hoover \cite{Hoover1979}]\label{AH-general}
Let $\mathfrak{X}$ be an exchangeable $\mathcal{L}$-structure, where $\mathcal{L}=\{R_1,\ldots,R_r\}$.
Then there exist Borel functions $f_1,\ldots,f_r$ so that the exchangeable structure $\mathfrak{X}^*$ generated by $f_1,\ldots,f_r$ satisfies $\mathfrak{X}\equalinlaw\mathfrak{X}^*$.
\end{theorem}

We can decompose the structure $\mathfrak{X}^*$ into the structures $\mathfrak{X}^*_\alpha$ given by specifying $\xi_\emptyset=\alpha$ for each $\alpha\in[0,1]$:
\[\vec x\in R^{\mathcal{X}^*_\alpha}_i\quad\Longleftrightarrow\quad f_i(\alpha,(\xi_{s})_{\emptyset\subsetneq s\subseteq\rng\vec x},(\prec_{\vec y})_{\vec y\subseteq\vec x})=1.\]
Each $\mathfrak{X}^*_\alpha$ is \emph{dissociated} (Definition \ref{defn:dissociated}) and, thus, Theorem \ref{AH-general} affords the interpretation of arbitrary exchangeable structures as mixtures of dissociated exchangeable structures.  For the rest of this section, we assume $\mathfrak{X}$ is dissociated.

\begin{example}
One of the simplest interesting examples of an exchangeable,  dissociated structure is the {\em random graph} $\mathcal{X}$, which is defined by putting an edge between each pair $(n,m)$ according to the outcome of independent fair coin flips. 
In terms of the Aldous--Hoover theorem, the function $f$ can be chosen so that it depends only on the $\xi_{\{x,y\}}$ component.
The random graph is a standard example of an important family of structures---the \emph{Fra\"isse limits}---which play a key role in our general theory; see Sections \ref{section:summary} and \ref{section:proofs}.
\end{example}

\begin{example}
  An example illustrating what happens in the case of asymmetric relations is the random \emph{tournament}.  Recall that a tournament is a total directed graph such that between any two vertices there is exactly one directed edge.  The random tournament depends only on the uniform random ordering $\prec_{(x,y)}$, where the edge points from $x$ to $y$ if and only if $x\prec_{\{x,y\}}y$.
\end{example}

Any probability measure $\mu$ on $\lN$ induces a probability measure $\mu_n$ on $\mathcal{L}_{[n]}$ by
\[\mu_n(\mathfrak{S}):=\mu(\{\mathfrak{M}\in\lN:\,\mathfrak{M}|_{[n]}=\mathfrak{S}\}),\quad\mathfrak{S}\in\ln.\]
Define the {\em support} of $\mu$ as the set of finite structures for which $\mu_n$ is positive:
\[\text{support}(\mu):=\bigcup_{n\geq1}\{\mathfrak{S}\in\ln:\,\mu_n(\mathfrak{S})>0\}.\]
Since there are countably many finite subsets of $\Nb$,  every finite substructure of $\mathfrak{X}\sim\mu$ is isomorphic to a structure in support($\mu$) with probability 1. 
 That is, the \emph{age} of $\mathfrak{X}$ is, with probability 1, contained in the support of $\mu$.
\begin{definition}
The {\em age} of $\mathfrak{M}$, denoted $\age(\mathfrak{M})$, is the set of all finite $\mathcal{L}$-structures embedded in $\mathfrak{M}$.
That is,
\[\age(\mathfrak{M}):=\{\mathfrak{S}\in\bigcup_{n\in\Nb}\ln:\,\mathcal{H}(\mathfrak{S},\mathfrak{M})\text{ is non-empty}\},\]
where 
\[\mathcal{H}(\mathfrak{S},\mathfrak{M}):=\{\text{embeddings }\phi:\mathfrak{S}\rightarrow\mathfrak{M}\}\]
is the set of embeddings of $\mathfrak{S}$ in $\mathfrak{M}$.
\end{definition}
(Note that the usual definition of $\age(\mathfrak{M})$ includes all finite substructures of $\mathfrak{M}$, perhaps identified if they are isomorphic.  It is convenient for our purposes to specify the universe of our structures precisely as $[1,n]$.
Among other features, this ensures that structures in the age are canonically ordered.)

If $\mu_n(\mathfrak{S})>0$, then exchangeability and dissociation of $\mathfrak{X}$ imply that, with probability 1, there exists a finite set $S$ so that $\mathfrak{X}|_S$ is isomorphic to $\mathfrak{S}$.  (In fact, such sets occur with frequency $\mu_n(\mathfrak{S})$ with probability 1.)  
In the exchangeable and dissociated case, $\age(\mathfrak{X})$ and $\text{support}(\mu)$ coincide with probability 1. 
 In particular, there is a unique collection of finite structures determined by $\mu$ such that, with probability 1, $\age(\mathfrak{X})$ is equal to this collection.

Since $\mathfrak{X}$ is dissociated, $\age(\mathfrak{X})$ also satisfies the \emph{joint embedding property} with probability 1.
\begin{definition}
  A collection of finite structures $K$ has the {\em joint embedding property} (JEP) if for all $\mathfrak{S},\mathfrak{T}\in K$ there exists $\mathfrak{U}\in K$ such that $\mathfrak{S}$ and $\mathfrak{T}$ are embedded in $\mathfrak{U}$.
\end{definition}

Our main theorems require $\mathfrak{M}$ to have additional properties.

\begin{definition}[Ultrahomogeneity]\label{def:ultra}
An $\mathcal{L}$-structure $\mathfrak{M}$ is \emph{ultrahomogeneous} if every embedding $\phi:\mathfrak{N}\rightarrow\mathfrak{M}$, with $|\mathfrak{N}|\subseteq|\mathfrak{M}|$ finite, extends to an automorphism $\overline{\phi}:\mathfrak{M}\rightarrow\mathfrak{M}$.
\end{definition}

The following establishes the equivalence between our definition of relative exchangeabiity and its weaker form (see (ii) below).
By Proposition \ref{prop:equiv} our main theorems immediately generalize to this case.
\begin{prop}\label{prop:equiv}
Let $\mathfrak{M}$ be ultrahomogeneous and let $\mathfrak{X}$ be a random $\mathcal{L}'$-structure.
The following are equivalent.
\begin{itemize}
	\item[(i)] $\mathfrak{X}$ is relatively exchangeable with respect to $\mathfrak{M}$ as in Definition \ref{defn:relative exchangeability}.
	\item[(ii)] $\mathfrak{X}^{\sigma}\equalinlaw\mathfrak{X}$ for every automorphism $\sigma:\mathfrak{M}\to\mathfrak{M}$.
\end{itemize}
\end{prop}

\begin{proof}
That (i) implies (ii) is automatic.  In the reverse direction, suppose $\mathfrak{X}$ satisfies (ii) and let $\phi:S\to T$ be an embedding $\mathfrak{M}|_S\to\mathfrak{M}$.
Since $\mathfrak{M}$ is ultrahomogeneous, there is an automorphism $\bar{\phi}:\mathfrak{M}\to\mathfrak{M}$ such that $\bar{\phi}\upharpoonright S=\phi$.
By (ii), $\mathfrak{X}^{\bar{\phi}}\equalinlaw\mathfrak{X}$ and, in particular, $\mathfrak{X}^{\bar{\phi}}|_{S'}\equalinlaw\mathfrak{X}|_{S'}$ for all finite $S'\subseteq\Nb$.
Since $\bar{\phi}$ extends $\phi$, it follows that $\mathfrak{X}|_S\equalinlaw\mathfrak{X}^{\bar{\phi}}|_{T}=\mathfrak{X}|_T^{\phi}$, establishing (i).
\end{proof}

Suppose that $\mathfrak{X}$ is ultrahomogeneous with probability 1.  A standard back-and-forth argument shows that there is a single structure $\mathfrak{M}$ (up to isomorphism) such that, with probability 1, $\mathfrak{X}$ is isomorphic to $\mathfrak{M}$, and \cite[Theorem 1.1]{AFP} implies that $\age(\mathfrak{X})$ exhibits \textit{disjoint amalgamation}.
\begin{definition}[Disjoint amalgamation]\label{def:DAP}
A collection of finite structures $K$ has the {\em disjoint amalgamation property}\footnote{This is often called the \emph{strong amalgamation property}, as in \cite{AFP}.  We follow the authors who prefer ``disjoint'' on the grounds that ``strong'' is an overused adjective.} (DAP) 
 if for any $\mathfrak{S},\mathfrak{T},\mathfrak{T}'\in K$ and embeddings $\phi:\mathfrak{S}\rightarrow\mathfrak{T}$ and $\phi':\mathfrak{S}\rightarrow\mathfrak{T}'$ there exists a structure $\mathfrak{U}\in K$ and embeddings $\psi:\mathfrak{T}\rightarrow\mathfrak{U}$ and $\psi':\mathfrak{T}'\rightarrow\mathfrak{U}$ such that $\psi\circ\phi=\psi'\circ\phi'$ and $\im(\psi\circ\phi)=\im(\psi)\cap\im(\psi')$, where $\im(\phi):=\{t\in|\mathfrak{T}|:\,\exists s\in|\mathfrak{S}|\,(\phi(s)=t)\}$ is the {\em image} of $\phi$.
We often abuse the terminology slightly and say $\mathfrak{M}$ has DAP when $\age(\mathfrak{M})$ has DAP.
\end{definition}

Disjoint amalgamation implies that any pair of structures $\mathfrak{T},\mathfrak{T}'$ can be amalgamated into a larger structure without identifying any elements that are not already identified.
In the presence of ultrahomogeneity together with our restriction to languages with only relation symbols, disjoint amalgamation is equivalent to the trivial definable closure property mentioned in the introduction.  For our purposes, DAP is the more useful characterization to work with.

Ackerman, Freer \& Patel \cite[Corollary 1.3]{AFP} show that ultrahomogeneity and the disjoint amalgamation property for $\mathfrak{M}$ imply the existence of an exchangeable random structure that is almost surely isomorphic to $\mathfrak{M}$.

\begin{theorem}[Ackerman, Freer \& Patel \cite{AFP}]\label{AFP_theorem}
 Suppose $\mathfrak{M}$ is ultrahomogeneous and $\age(\mathfrak{M})$ satisfies the disjoint amalgamation property.  Then there is a dissociated, exchangeable random structure $\mathfrak{X}$ such that $\mathfrak{X}$ is isomorphic to $\mathfrak{M}$ with probability 1.
Moreover, there exist Borel functions $f_1,\ldots,f_r$ such that $\mathfrak{X}$ can be generated as the exchangeable structure of the form
\begin{equation}\label{eq:random-free}\vec x\in R^{\mathfrak{X}}_i\quad\Longleftrightarrow\quad f_i(\xi_{x_1},\ldots,\xi_{x_{\ar(R_i)}})=1,\end{equation}
for $(\xi_x)_{x\in\Nb}$ i.i.d.\ Uniform$[0,1]$ random variables.
\end{theorem}

\subsection{$n$-DAP}

Our strongest results require an amalgamation property for all $n\geq1$ simultaneously.
\begin{definition}
Let $K$ be a collection of finite structures that is closed under isomorphism.
For $n\geq1$, we say that $K$ satisfies the \emph{$n$-disjoint amalgamation property} ($n$-DAP) if for every collection $(\mathfrak{S}_i)_{1\leq i\leq n}$ of structures satisfying $\mathfrak{S}_i\in K$, $|\mathfrak{S}_i|=[n]\setminus\{i\}$, and $\mathfrak{S}_i|_{[n]\setminus\{i,j\}}=\mathfrak{S}_j|_{[n]\setminus\{i,j\}}$ for all $1\leq i,j\leq n$ there exists $\mathfrak{S}\in K$ with $|\mathfrak{S}|=[n]$ such that $\mathfrak{S}|_{[n]\setminus\{i\}}=\mathfrak{S}_i$ for every $1\leq i\leq n$.
Again, we say $\mathfrak{M}$ has $n$-DAP when $\age(\mathfrak{M})$ does.
\end{definition}

Under $n$-disjoint amalgamation, if we specify a structure on each proper subset of $[n]$ in a way that is pairwise compatible, then there is a way to unify these structures into a single structure on all of $n$.  By slight abuse of terminology, if $K$ is a collection of finite structures not closed under isomorphism (like $\age(\mathfrak{M})$ as defined above), we say $K$ has $n$-DAP if the closure of $K$ under isomorphism has $n$-DAP.
When $K$ is closed under substructures (as all our classes will be), $2$-DAP is equivalent to DAP.

There is a simpler condition on the theory $\mathcal{T}$ which implies $n$-DAP for all $n\geq1$ and is satisfied by the most common examples.

\begin{definition}[Parametric universal theories]\label{defn:local}
$\mathcal{T}$ is a \emph{parametric universal theory} if each sentence in $\mathcal{T}$ has the form 
\[\forall x_1,\ldots,x_k\ \phi(x_1,\ldots,x_k),\] where $\phi$ is quantifier-free and every atomic formula in $\phi$ contains all $k$ variables $x_1,\ldots,x_k$.
\end{definition}

\begin{lemma}
  If $\mathcal{T}$ is a parametric universal theory with models of every finite size and $K$ is the set of finite models of $\mathcal{T}$ then $K$ satisfies $n$-DAP for all $n\geq1$.
\end{lemma}

\begin{proof}
Consider some $n$ and suppose that for each $i$ we have $\mathfrak{S}_i$ with $|\mathfrak{S}_i|=[n]\setminus\{i\}$ so that $\mathfrak{S}_i|_{[n]\setminus\{i,j\}}=\mathfrak{S}_j|_{[n]\setminus\{i,j\}}$ for $1\leq i,j\leq n$.  We define a structure $\mathfrak{S}$ on $[1,n]$.  For any tuple $\vec x$ such that fewer than $n$ elements appear in $\vec x$, choose $i$ not appearing in $\vec x$ and set $\vec x\in R^{\mathfrak{S}}_k$ if and only if $\vec x\in R^{\mathfrak{S}_i}_k$.  Note that this definition does not depend on our choice of $i$, since if $j$ also does not appear in $\vec x$ then $\rng\vec x\subseteq[n]\setminus\{i,j\}$, so $\vec x\in R^{\mathfrak{S}_i}_k$ if and only if $\vec x\in R^{\mathfrak{S}_j}_k$.

We then choose an arbitrary structure $\mathfrak{T}$ in $K$ of size $n$; without loss of generality, we assume $|\mathfrak{T}|=[n]$.  For each sequence $\vec x$ containing all $n$ elements of $[n]$, we set $\vec x\in R_k^{\mathfrak{S}}$ if and only if $\vec x\in R_k^{\mathfrak{T}}$.

Consider some axiom $\forall x_1,\ldots,x_k\phi(x_1,\ldots,x_k)$ from $\mathcal{T}$.  For any $k$-tuple $\vec x=x_1,\ldots,x_k$, if fewer than $n$ distinct elements appear in $\vec x$ then $\vec x$ is contained in the universe of some $\mathfrak{S}_i$, and the axiom is satisfied because it is satisfied for each $\mathfrak{S}_i$.  If $\vec x$ contains all $n$ elements then every atomic formula in $\phi$ contains all $n$ elements, so the axiom is satisfied because it is satisfied in $\mathfrak{T}$.

\end{proof}

\begin{example}\label{ex:nDAP}
Graphs and hypergraphs are specified by parametric universal theories but equivalence relations are not.
In general, a graph $\mathfrak{M}$ consists of a single binary relation $R^{\mathfrak{M}}$ satisfying the empty theory, which is trivially parametric.
If self-loops are forbidden, then $\mathfrak{M}$ is {\em anti-reflexive}:
\begin{equation}\label{eq:anti-reflexive}\mathfrak{M}\vDash\forall x\,(x,x)\notin R.\end{equation}
An {\em undirected graph} $\mathfrak{M}$ satisfies the further symmetry property
\begin{equation}\label{eq:symmetry}
\mathfrak{M}\vDash \forall x,y\, ((x,y)\in R\rightarrow(y,x)\in R).\end{equation}
Both \eqref{eq:anti-reflexive} and \eqref{eq:symmetry} are parametric universal sentences because \eqref{eq:anti-reflexive} consists of a single atomic sentence and \eqref{eq:symmetry} can be written as
\[\forall x,y\ ((x,y)\in R\wedge (y,x)\in R)\vee((x,y)\notin R\wedge(y,x)\notin R).\]

On the other hand, an equivalence relation $\mathfrak{M}$ is a binary relation $R^{\mathfrak{M}}$ that satisfies the {\em transitivity axiom}
\begin{equation}\label{eq:transitivity}
\mathfrak{M}\vDash\forall x,y,z\ ((x,y)\in R\wedge (y,z)\in R)\to (x,z)\in R,
\end{equation}
which consists of the three atomic sentences
\[(x,y)\in R,\,(y,z)\in R,\,\text{and }(x,z)\in R,\]
none containing all the variables $x,y,z$.
Furthermore, the class of finite equivalence relations does not satisfy $n$-DAP for all $n\geq3$.
Let $K$ be the set of all finite equivalence relations.
Take $n=3$ and define each $\mathfrak{S}_i$ by its equivalence classes $C_1/C_2/\cdots$: $\mathfrak{S}_1=\{2\}/\{3\}$, $\mathfrak{S}_2=\{1,3\}$, and $\mathfrak{S}_3=\{1,2\}$.
Then $\mathfrak{S}_{i}|_{[n]\setminus\{i,j\}}=\mathfrak{S}_j|_{[n]\setminus\{i,j\}}$ for all $i$ and $j$ but there is no equivalence relation $\mathfrak{S}$ of $[n]$ such that $\mathfrak{S}|_{[n]\setminus\{i\}}=\mathfrak{S}_i$ for every $i=1,2,3$.

\end{example}

\section{Summary of Results}\label{section:summary}

Our main theorems generalize Aldous--Hoover and related results to characterize the probability law of random structures $\mathfrak{X}$ that are relatively exchangeable with respect to a structure $\mathfrak{M}$ with trivial definable closure and ultrahomogeneity.
Stronger assumptions about the structure of $\mathfrak{M}$ elicit a stronger representation for $\mathfrak{X}$.

\subsection{The Strongest Representation}

The notion of an exchangeable structure generated by functions generalizes to $\mathfrak{M}$-exchangeable structures.
\begin{definition}
  Let $\mathcal{L}=\{Q_1,\ldots,Q_r\}$ and $\mathcal{L}'=\{R_1,\ldots,R_{r'}\}$ be languages so that each $R_i$ has $\ar(R_i)\leq k$ and let $f_1,\ldots,f_{r'}$ be Borel functions.  The \emph{$\mathfrak{M}$-exchangeable structure generated by} $f_1,\ldots,f_{r'}$ is the structure $\mathfrak{X}^*=(\mathbb{N},R_1^{\mathfrak{X}^*},\ldots,R_{r'}^{\mathfrak{X}^*})$ given by choosing $(\xi_s)_{s\subseteq\mathbb{N}:\,|s|\leq k}$ i.i.d.\ Uniform$[0,1]$ and $(\prec_s)_{s\subseteq\mathbb{N}:\,|s|\leq k}$ independent uniform random orderings and putting
\[\vec x\in R_i^{\mathfrak{X}^*}\quad\Longleftrightarrow\quad f_i(\mathfrak{M}|_{\rng\vec x},(\xi_{s})_{s\subseteq\rng\vec x},(\prec_{\vec y})_{\vec y\subseteq \vec x})=1.\]
\end{definition}

We obtain the following generalization of the de Finetti--Aldous--Hoover theorem to arbitrary relatively exchangeable structures.
\begin{theorem}\label{thm:nice1}
Let $\mathcal{L},\mathcal{L}'$ be signatures and $\mathfrak{M}$ be a countable $\mathcal{L}$-structure that is ultrahomogeneous and has $n$-DAP for all $n\geq1$.  Without loss of generality, assume $|\mathfrak{M}|=\Nb$.  

Let $\mathcal{L}'=\{R_1,\ldots,R_{r'}\}$ have $\ar(R_i)\leq k$ for all $1\leq i\leq r'$.  Suppose $\mathfrak{X}$ is a random $\mathcal{L}'$-structure that is relatively exchangeable with respect to $\mathfrak{M}$.
Then there exist Borel functions $f_1,\ldots,f_{r'}$ such that $\mathfrak{X}\equalinlaw \mathfrak{X}^*$, where $\mathfrak{X}^*$ is the $\mathfrak{M}$-exchangeable structure generated by $f_1,\ldots,f_{r'}$.
\end{theorem}

\begin{lemma}\label{thm:nice_dissociated}
  If, in the situation of Theorem \ref{thm:nice1}, $\mathfrak{X}$ is also dissociated, then $f_1,\ldots,f_{r'}$ can be chosen so that the $f_i$ do not depend on $\xi_\emptyset$.
\end{lemma}
We will prove these in Section \ref{sec:proof_nice}.

\begin{remark}
 If $\mathcal{L}$ is the empty language and $\mathcal{L}'$ consists of a single $k$-ary relation, then Theorem \ref{thm:nice1} specializes to de Finetti's theorem \cite{deFinetti1937} (when $k=1$) and the Aldous--Hoover theorem \cite{Aldous1981,Hoover1979} (when $k>1$).
\end{remark}

\begin{remark}
  By the same argument, the analogous statement would hold if we replace $\mathfrak{X}$ with a Borel-valued structure---that is, if each $R_i^{\mathfrak{X}}$ is a function from $\mathbb{N}^{k_i}$ to some Borel space $\Omega_i$.  The analogous statement would then give a representation where $f_1,\ldots,f_{r'}$ are Borel-measurable functions so 
\[R_i^{\mathfrak{X}}(\vec x)=f_i(\mathfrak{M}|_{\rng\vec x},(\xi_s)_{s\subseteq\rng\vec x},(\prec_{\vec y})_{\vec y\subseteq\vec x}).\]
The case where we consider structures is precisely the case where each $\Omega_i=\{0,1\}$.
\end{remark}

The main outcome of Theorem \ref{thm:nice1} is that when $\mathfrak{M}$ is ultrahomogeneous and has $n$-DAP for all $n\geq1$, an $\mathfrak{M}$-exchangeable structure $\mathfrak{X}$ can be generated so that, for every subset $S\subseteq\Nb$, $\mathfrak{X}|_{S}$ depends only on the smallest non-trivial substructure of $\mathfrak{M}$, namely $\mathfrak{M}|_{S}$.
Some examples show how the representation depends on these assumptions.
Theorem \ref{thm:main1} covers the case where $n$-DAP for all $n\geq1$ is relaxed to 2-DAP.

\begin{example}\label{example:strong-rep}
Suppose $\mathfrak{M}$ and $\mathfrak{X}$ are both subsets of $\Nb$---that is, $\mathcal{L}$ and $\mathcal{L}'$ each contain a single unary relation---and $\mathfrak{M}=(\Nb,\mathcal{P})$ is a model with $\mathcal{P}\subseteq\Nb$ infinite and coinfinite.
If $S,T$ are finite subsets of $\Nb$, an embedding $\phi:S\rightarrow T$ that preserves $\mathfrak{M}$ must map $S\cap \mathcal{P}$ to $T\cap\mathcal{P}$ and $S\setminus\mathcal{P}$ to $T\setminus\mathcal{P}$.  
Thus, an $\mathfrak{M}$-exchangeable $\mathfrak{X}$ can be viewed as two separate exchangeable structures---one on $\mathcal{P}$ and one on $\Nb\setminus\mathcal{P}$.

Theorem \ref{thm:nice1} says that $\mathfrak{X}=(\Nb,\mathcal{X})$ can be represented by
\begin{equation}\label{eq:2-coloring}
n\in\mathcal{X}\quad\Longleftrightarrow\quad f(\mathfrak{M}|_{\{n\}},\xi_\emptyset, \xi_{\{n\}})=1,\end{equation}
for i.i.d.\ Uniform$[0,1]$ random variables $\{\xi_{\emptyset};(\xi_{\{n\}})_{n\geq1}\}$.
Thus, the event $n\in\mathcal{X}$ depends on three things: a global random variable $\xi_\emptyset$, a random variable specific to $n$, and whether or not $n\in\mathcal{P}$.

The natural way to interpret \eqref{eq:2-coloring} is that we have a probability measure $\Theta$ on $[0,1]^2$ from which we choose $(\theta_0,\theta_1)$.  
Given $(\theta_0,\theta_1)$, we determine $\mathcal{X}$ by independently flipping a coin for each $n$: if $n\in\mathcal{P}$, we flip a coin with probability $\theta_1$ of landing heads; otherwise, we flip a coin with probability $\theta_0$ of landing heads.  The random variable $\xi_\emptyset$ corresponds to the choice of $(\theta_0,\theta_1)$, $\mathfrak{M}|_{\{n\}}$ determines which coin to flip for each $n$, and $\xi_{\{n\}}$ determines the outcome of the coin flip associated to $n$.
The representation of this special case has been shown before by one of the authors \cite{Crane2014AOP}.

Note that $\mathfrak{X}$ is not (necessarily) dissociated---$n\in\mathcal{X}$ and $n'\in\mathcal{X}$ are not independent since both depend on the same random choice of $(\theta_0,\theta_1)$---and so $f$ depends non-trivially on $\xi_\emptyset$.
\end{example}

\begin{example}\label{example:graph_on_graph}
Suppose $\mathcal{L}$ and $\mathcal{L}'$ each contain a single binary relation and $\mathfrak{M}=(\Nb,\mathcal{R})$ is a copy of the random graph (i.e., the unique up to isomorphism universal ultrahomogeneous countable graph).  If $\mathfrak{X}=(\Nb,\mathcal{X})$ is $\mathfrak{M}$-exchangeable, Theorem \ref{thm:nice1} gives a representation
\[(n,m)\in\mathcal{X}\quad\Longleftrightarrow\quad f(\mathfrak{M}|_{\{n,m\}},\xi_\emptyset,\xi_{\{n\}},\xi_{\{m\}},\xi_{\{n,m\}},\prec_{(n,m)})=1.\]
\end{example}

The next two examples fail $n$-DAP and illustrate why we cannot drop that requirement from the statement of the theorem.
\begin{example}\label{example:weak-rep}
Let $\mathcal{L}$ contain a single $3$-ary relation and $\mathcal{L}'=\{S\}$ contain a single binary relation.
Let $\mathfrak{M}=(\Nb,\mathcal{R})$ be an $\mathcal{L}$-structure such that for each $i$, $\mathcal{R}(i)=\{(j,k)\mid (i,j,k)\in\mathcal{R}\}$ is an equivalence relation with exactly three equivalence classes, two infinite and the third consisting only of $i$.  We generate a random $\mathcal{L}'$-structure as follows.  For each $i\in\Nb$, we pick one of the two non-singleton equivalence classes of $\mathcal{R}(i)$ uniformly at random; let $\mathcal{B}^{*i}\subseteq\Nb$ be this equivalence class.  
We then put
\[(i,j)\in S^{\mathfrak{X}}\quad\Longleftrightarrow\quad j\in\mathcal{B}^{*i}.\]

By construction, $\mathfrak{X}$ is $\mathfrak{M}$-exchangeable and dissociated.  However, suppose we could find a representation
\begin{equation}\label{eq:nice-false}
(i,j)\in S^{\mathfrak{X}}\quad\Longleftrightarrow\quad f(\mathfrak{M}|_{\{i,j\}},\xi_{\{i\}},\xi_{\{j\}},\xi_{\{i,j\}},\prec_{(i,j)})=1.\end{equation}
Take a triple not in $\mathcal{R}$; without loss of generality suppose $(1,2,3)\not\in\mathcal{R}$, so that $2$ and $3$ are in different equivalence classes of $\mathcal{R}(1)$.  Then we must have exactly one of $(1,2)$ and $(1,3)$ in $S^{\mathfrak{X}}$.  With probability $1/2$, $\prec_{(1,2)}=\prec_{(1,3)}$.  But since $\mathfrak{M}|_{\{1,2\}}=\mathfrak{M}|_{\{1,3\}}$, so representation \eqref{eq:nice-false} implies $(1,2)\in S^{\mathfrak{X}}$ and $(1,3)\in S^{\mathfrak{X}}$ are conditionally independent given $\xi_{\{1\}}$ and $\prec_{(1,2)}=\prec_{(1,3)}$.
In particular, if there is a non-zero probability that $(1,2)\in S^{\mathfrak{X}}$ then there is a non-zero probability that both $(1,2)$ and $(1,3)$ are in $S^{\mathfrak{X}}$.

Notice that $\mathfrak{M}$ does not have $3$-DAP: suppose we try to build a structure containing four elements $\{1,2,3,4\}$ so $(1,2,3),(1,2,4)\in \mathcal{R}$ but $(1,3,4)\not\in\mathcal{R}$.  The restriction to each three element subset gives an element of $\age(\mathfrak{M})$, but they are incompatible as a four element subset.
\end{example}


The next example fails $n$-DAP despite having no definable equivalence relations\footnote{The underlying model-theoretic example is a structure without $n$-DAP which is a reduct of a structure with $n$-DAP.  We thank A.\ Kruckman for calling our attention to this example from MacPherson \cite{Macpherson2010}.}:
\begin{example}
  Let $\mathcal{L}_0$ consist of a single binary relation $R_0$.  Let $\mathfrak{M}_0$ be the $\mathcal{L}_0$-structure which interprets $R_0$ as the random graph.  Let $\mathcal{L}$ consist of a single $3$-ary relation $R$, and let $R^{\mathfrak{M}}$ consist of those triples $(x,y,z)$ of distinct elements such that $|R_0^{\mathfrak{M}_0}\cap[\{x,y,z\}]^2|$ is odd.  $\mathfrak{M}$ is clearly an undirected hypergraph, and it can be checked that it is universal subject to the constraint that when $\{w,x,y,z\}$ are distinct, $|R^{\mathfrak{M}}\cap[\{w,x,y,z\}]^3|$ is even.  In particular, $\mathfrak{M}$ is ultrahomogeneous, but fails to have $4$-DAP.

Let $\mathcal{L}'$ consist of a single binary relation $S$.  We begin by defining an $\mathfrak{M}_0$-exchangeable $\mathcal{L}'$-structure $\mathfrak{X}$: for each vertex $x$, we choose $\xi_x\in\{0,1\}$ i.i.d.  We define $(x,y)\in S^{\mathfrak{X}}$ if either $\xi_x\neq\xi_y$ and $(x,y)\in R_0^{\mathfrak{M}_0}$ or $\xi_x=\xi_y$ and $(x,y)\not\in R_0^{\mathfrak{M}_0}$.  Notice that $\mathfrak{X}$ is dissociated.

Then $\mathfrak{X}$ is actually $\mathfrak{M}$-exchangeable.  To see this, suppose $\phi:\mathfrak{M}|_S\rightarrow\mathfrak{M}|_T$ is an isomorphism.  Then $R_0^{\mathfrak{M}}$ and $\phi(R_0^{\mathfrak{M}}\cap S^2)$ induce two graphs on $T$; write $E$ for the symmetric difference---that is, $E$ is those edges $(x,y)\in [T]^2$ such that either $(x,y)\in R_0^{\mathfrak{M}}$ but $(\phi^{-1}(x),\phi^{-1}(y))\not\in R_0^{\mathfrak{M}}$, or vice versa.  Since $\mathfrak{M}|_S$ and $\mathfrak{M}|_T$ are isomorphic, for every triple $\{x,y,z\}\subseteq T$ of distinct elements, $|E\cap[\{x,y,z\}]^2|$ must be even.  Choose any vertex $v\in T$ and let $V$ be the set containing every vertex which is \emph{not} adjacent to $v$ in $E$ (including $v$).  Then $V$ intersects every edge in $E$ exactly once: if $(x,y)\in E$ then either exactly one of these vertices is $v$, or the triple $\{v,x,y\}$ has an even number of edges, so $(x,y)$ is one and either $(v,x)$ or $(v,y)$ is the other, so exactly one of $x$ and $y$ belongs to $V$.  But now we see that for any choice of values $\xi_x$ giving us a structure $\mathfrak{X}|_S$, by flipping the values on those $x\in V$, we get the same structure on $\mathfrak{X}|_T$.  This shows that $\mathfrak{X}$ is $\mathfrak{M}$-exchangeable.

But suppose we could represent $\mathfrak{X}$ in the form
\[(i,j)\in S^{\mathfrak{X}}\quad\Longleftrightarrow\quad f(\mathfrak{M}|_{\{i,j\}},\xi_{\{i\}},\xi_{\{j\}},\xi_{\{i,j\}},\prec_{(i,j)})=1.\]
Since $\mathfrak{M}$ restricted to a pair is trivial, this really has the form
\[(i,j)\in S^{\mathfrak{X}}\quad\Longleftrightarrow\quad f(\xi_{\{i\}},\xi_{\{j\}},\xi_{\{i,j\}},\prec_{(i,j)})=1.\]
Then $\mathfrak{X}$ must be fully exchangeable.  But this is not the case; for instance, if $(x,y,z)\in R^{\mathfrak{M}}$ then $|S^{\mathfrak{X}}\cap[\{x,y,z\}]^2|$ is even while if $(x,y,z)\not\in R^{\mathfrak{M}}$ then $|S^{\mathfrak{X}}\cap[\{x,y,z\}]^2|$ is odd (consider the four possible values of $\xi_x+\xi_y+\xi_z$ by cases).
\end{example}

\subsection{Age Indexed Processes}\label{section:canonical structures}

Part of our motivation for considering ultrahomogeneous $\mathfrak{M}$ with disjoint amalgamation is that these structures have a nice universality property: if $\mathfrak{N}$ is a countable structure with $\age(\mathfrak{N})\subseteq\age(\mathfrak{M})$ then there is  an embedding of $\mathfrak{N}$ into $\mathfrak{M}$.

For any $\mathfrak{S}\in\age(\mathfrak{M})$ with $|\mathfrak{S}|=[n]$, there is a natural embedding $\rho_{\mathfrak{S},\mathfrak{M}}:\mathfrak{S}\rightarrow\mathfrak{M}$ obtained by successively choosing $\rho_{\mathfrak{S},\mathfrak{M}}(i)=m_i$, where $m_i$ is least so that $\rho_{\mathfrak{S},\mathfrak{M}}\upharpoonright [i]$ is an embedding $\mathfrak{S}|_{[i]}\rightarrow\mathfrak{M}$ for every $1\leq i\leq n$.  (Here we use the fact that we have defined $\age(\mathfrak{M})$ to contain only structures with universe $[n]$ for some $n$.)

If $\mathfrak{X}$ is $\mathfrak{M}$-exchangeable, it induces a family of finite random structures as follows.

\begin{definition}[Age indexed random structures]\label{defn:canonical structure}
Let $\mathcal{L},\mathcal{L}'$ be signatures, $\mathfrak{M}\in\lN$ be ultrahomogeneous and satisfy DAP, and $\mathfrak{S}\in\age(\mathfrak{M})$.  Suppose that for each $\mathfrak{S}\in\age(\mathfrak{M})$ we have a random $\mathcal{L}'$-structure on $|\mathfrak{S}|$ such that whenever $\phi:\mathfrak{S}\rightarrow\mathfrak{T}$ is an embedding, $\mathfrak{X}^{\mathfrak{S}}\equalinlaw(\mathfrak{X}^{\mathfrak{T}})^\phi$.  Then we call $\{\mathfrak{X}^{\mathfrak{S}}\}_{\mathfrak{S}\in\age(\mathfrak{M})}$ an \emph{age indexed random $\mathcal{L}'$-structure}.

When $\mathfrak{X}$ is an $\mathfrak{M}$-exchangeable $\mathcal{L}'$-structure, we define a random $\mathcal{L}'$-structure $\mathfrak{X}^{\mathfrak{S}}=\mathfrak{X}^{\rho_{\mathfrak{S},\mathfrak{M}}}$, where $\rho_{\mathfrak{S},\mathfrak{M}}:\mathfrak{S}\rightarrow\mathfrak{M}$ is the natural embedding defined above.  That is, each $\mathfrak{X}^{\mathfrak{S}}$ is the finite random $\mathcal{L}'$-structure induced by the image of $\mathfrak{S}$ in $\mathfrak{M}$.
\end{definition}
In light of the following proposition, we call $\{\mathfrak{X}^{\mathfrak{S}}\}_{\mathfrak{S}\in\age(\mathfrak{M})}$ the \emph{age indexed random $\mathcal{L}'$-structure induced by} $\mathfrak{X}$.

\begin{prop}
Let $\mathfrak{X}$ be $\mathfrak{M}$-exchangeable, ${\mathfrak{S}},\mathfrak{T}\in\age(\mathfrak{M})$, and $\phi:\mathfrak{S}\rightarrow\mathfrak{T}$ be any embedding.
Then 
\[\mathfrak{X}^{{\mathfrak{S}}}\equalinlaw (\mathfrak{X}^{{\mathfrak{T}}})^{\phi}.\]
\end{prop}

\begin{proof} 
Let $\mathfrak{S},\mathfrak{T}\in\age(\mathfrak{M})$, $\rho_{{\mathfrak{S}},\mathfrak{M}},\rho_{{\mathfrak{T}},\mathfrak{M}}$ be the natural embeddings defined above, and assume $\phi:\mathfrak{S}\rightarrow\mathfrak{T}$ is an embedding.
Then $\rho_{{\mathfrak{T}},\mathfrak{M}}\circ\phi:\mathfrak{S}\rightarrow\mathfrak{M}$ is also an embedding and 
\[\mathfrak{X}^{{\mathfrak{S}}}\equalinlaw \mathfrak{X}|_S\equalinlaw  \mathfrak{X}^{\rho_{{\mathfrak{T}},\mathfrak{M}}\circ\phi}\equalinlaw (X^{{\mathfrak{T}}})^{\phi}.\]
The proof is complete.
\end{proof}

Conversely, whenever $\{\mathfrak{X}^{\mathfrak{S}}\}_{\mathfrak{S}\in\age(\mathfrak{M})}$ is an age indexed random structure, we can construct an $\mathfrak{M}$-exchangeable random $\mathcal{L}'$-structure sequentially through its finite restrictions $(\mathfrak{X}|_{[n]})_{n\geq0}$: We first choose $\mathfrak{X}|_{[0]}$ according to $\mathfrak{X}^{\mathfrak{M}|_{[0]}}$ and, given $\mathfrak{X}|_{[n]}$, we choose $\mathfrak{X}|_{[n+1]}$  according to $\mathfrak{X}^{\mathfrak{M}|_{[n+1]}}$ conditioned on $\mathfrak{X}^{\mathfrak{M}|_{[n+1]}}|_{[n]}=\mathfrak{X}|_{[n]}$.
The upshot of Theorem \ref{thm:main1} is that this construction is always possible for $\mathfrak{M}$-exchangeable structures, as long as $\mathfrak{M}$ is ultrahomogeneous and has 2-DAP.
We prove this by first constructing a {potential age indexed structure}.

\begin{definition}[Potential age indexed structures]
  Let $f_1,\ldots,f_{r'}$ be Borel functions and $\mathfrak{M}$ be an $\mathcal{L}$-structure.
The \emph{potential age indexed structure generated by} $f_1,\ldots,f_{r'}$ is the process $\{\mathfrak{X}^{\mathfrak{S}}\}_{\mathfrak{S}\in\age(\mathfrak{M})}$ given by choosing $(\xi_s)_{s\subseteq\mathbb{N}:\,|s|\leq k}$ i.i.d.\ Uniform$[0,1]$ and $(\prec_s)_{s\subseteq\mathbb{N}:\,|s|\leq k}$ independent uniform random orderings and setting for any $\mathfrak{S}\in\age(\mathfrak{M})$
\[\vec x\in R_i^{\mathfrak{X}^{\mathfrak{S}}}\quad\Longleftrightarrow\quad f_i(\mathfrak{S},(\xi_s)_{s\subseteq\rng\vec x},(\prec_{\vec y})_{\vec y\subseteq\vec x}).\]
\end{definition}
We call this a ``potential'' age indexed structure because it need not satisfy the invariance property of an age indexed structure.
\begin{definition}
  We call $f_1,\ldots,f_{r'}$ \emph{age compatible} if the potential age indexed structure generated by $f_1,\ldots,f_{r'}$ is an age indexed structure.
\end{definition}
Note that age symmetry is a property of the sequence of functions collectively---it is possible for $f_1$ and $f_2$ to be individually age compatible, but $(f_1,f_2)$ is not.  When $f_1,\ldots,f_{r'}$ are age compatible, the function $f_j$ that maps the tuples $(\xi_{s})_{s\subseteq\Nb:\ |s|\leq k},(\prec_s)_{s\subseteq\Nb:\ |s|\leq k}$ to a value \emph{does} depend on the labeling of ${\mathfrak{S}}$, but its distribution does not, as the following example illustrates.

\begin{example}
A typical example that illustrates this is the age indexed structure corresponding to Example \ref{example:weak-rep}.
In this example, recall that $\mathcal{L}$ has a single $3$-ary relation and $\mathfrak{M}$ has the property that for each $i$, $\mathcal{R}(i)=\{(j,k)\mid (i,j,k)\in\mathcal{R}\}$ is an equivalence relation with two infinite equivalence classes, while $\mathcal{L}'$ has a single binary relation $S$.  
To generate the age indexed process, we define $f_1(\mathfrak{S}|_{[1,\max\{i,j\}]},\xi_{\{i\}},\xi_{\{j\}},\xi_{\{i,j\}})$ as follows: we ignore $\xi_{\{j\}}$ and $\xi_{\{i,j\}}$, and if $\xi_{\{i\}}<1/2$, we take $f_1(\mathfrak{S}|_{[1,\max\{i,j\}]},\xi_{\{i\}})=1$ if and only if $j$ is in same $\mathcal{R}(i)$ equivalence class as $1$, while if $\xi_{\{i\}}\geq 1/2$, we take $f_1(\mathfrak{S}|_{[1,\max\{i,j\}]},\xi_{\{i\}})=1$ if and only if $j$ is in a different $\mathcal{R}(i)$ equivalence class from $1$.  
(To avoid trivialities when $i=1$, we define $f_1$ in the analogous way according to whether $j$ is in the same $\mathcal{R}(1)$ equivalence class of $2$.)
Note that $\mathbb{P}(f_1(\mathfrak{S}|_{[1,\max\{i,j\}]},\xi_{\{j\}})=1)$ depends only on $\mathfrak{S}|_{\{i,j\}}$ (in this case the dependence is trivial, but in more complicated cases it need not be).
The more complicated dependence on the entire initial segment $\mathfrak{S}|_{[1,\max\{i,j\}]}$ tells us which values of $\xi_{\{j\}}$ correspond to which values of the function $f_1$.
\end{example}

We can now state a more general version of our main result, which drops the assumption that $\mathfrak{M}$ has $n$-DAP for all $n$.  
We prove this in Section \ref{sec:proof_main}.
\begin{theorem}\label{thm:main1}
Let $\mathcal{L}=\{Q_1,\ldots,Q_r\}$ and $\mathcal{L}'=\{R_1,\ldots,R_{r'}\}$ be signatures, let each $R_{i}$ have $\ar(R_{i})\leq k$, and let $\mathfrak{M}$ be a countable $\mathcal{L}$-structure that is ultrahomogeneous and whose age has the disjoint amalgamation property.  Without loss of generality, assume $|\mathfrak{M}|=\Nb$.

Suppose that $\mathfrak{X}$ is an $\mathfrak{M}$-exchangeable $\mathcal{L}'$-structure.
Then there exist age compatible Borel functions $f_1,\ldots,f_{r'}$ such that $\mathfrak{X}^{\mathfrak{S}}\equalinlaw\mathfrak{X}^{*,\mathfrak{S}}$ for every $\mathfrak{S}\in\age(\mathfrak{M})$ where $\mathfrak{X}^*$ is the age indexed structure generated by $f_1,\ldots,f_{r'}$.
\end{theorem}

Theorem \ref{thm:main1} drops the $n$-DAP assumption from Theorem \ref{thm:nice1}, but now $f_j$ depends on the entire finite structure $\mathfrak{S}$, not just on $\mathfrak{S}|_{\rng\vec x}$.  The corresponding $\mathfrak{M}$-exchangeable structure $\mathfrak{X}^*$ can then be constructed by
\begin{equation}\label{eq:rep1full}
\vec x\in R_j^{\mathfrak{X}^{*}}\quad\Longleftrightarrow\quad f_j(\mathfrak{M}|_{[\max \vec x]},(\xi_{s})_{s\subseteq\rng\vec x},(\prec_{\vec y})_{\vec y\subseteq\vec x})=1,\quad \vec x\in\Nb^{\ar(R_j)},
\end{equation}
for  $j=1,\ldots,r'$.
That is, we need to look at the entire structure up to $\max\vec x$, not just the substructure indexed by $\rng\vec x$.

\begin{remark}
  Again, the analogous statement holds, by the same argument, for Borel-valued structures.
\end{remark}

The representation in \eqref{eq:rep1full} yields a natural sequential construction for relatively exchangeable structures $\mathfrak{X}$ through their finite substructures $(\mathfrak{X}|_{[n]})_{n\geq0}$.
During the sequential construction, we need only keep track of the piece of $\mathfrak{M}$ we have ``seen'' so far, in the sense that when determining $\mathfrak{X}|_{[n+1]}$ based on $\mathfrak{X}|_{[n]}$, we need only consider $\mathfrak{M}|_{[n+1]}$.

Recall that our initial example of an exchangeable structure involved taking a large structure $\mathfrak{U}=(\Omega,\mathcal{R}_1,\ldots,\mathcal{R}_r)$ and a probability measure $\mu$ on $\Omega$ and sampling $\phi(1),\phi(2),\ldots$ i.i.d.\ from $\mu$ to obtain $\mathfrak{X}=\mathfrak{U}^{\phi}$.  The analogous procedure for choosing an $\mathfrak{M}$-exchangeable random structure entails taking a large $\mathcal{L}\cup\mathcal{L}'$-structure $\mathfrak{U}=(\Omega,\mathcal{Q}_1,\ldots,\mathcal{Q}_r,\mathcal{R}_1,\ldots,\mathcal{R}_{r'})$ and choosing points $\phi(n)\in\Omega$ subject to the constraint that the reduct $\mathfrak{U}^{\phi}\upharpoonright\mathcal{L}=(\Nb,\mathcal{Q}_1^{\phi},\ldots,\mathcal{Q}_r^{\phi})$ forms a model of $\mathfrak{M}$.  
The most natural approach is to choose points successively---first choose $\phi(0)$, then choose $\phi(1)$ subject to the constraints induced by the choice of $\phi(0)$, and so on.  The dependence of $f_j$ on an entire initial segment of $\mathfrak{M}$ reflects this procedure.

The next example demonstrates that trivial definable closure (without ultrahomogeneity) is not suficient to obtain the representation in Theorem \ref{thm:main1}.
\begin{example}\label{ex:TDC}
Let $\mathfrak{M}$ be the directed graph with an edge $(i,j)$ if and only if $j$ is odd and $j\neq i$.
We let $\mathfrak{X}=(\Nb,\mathcal{P})$ be a random unary relation such that, with probability $1/3$, $\mathcal{P}$ is the set of even integers and, with probability $2/3$, $\mathcal{P}$ contains each odd integer independently with probability $1/2$.
Since $\mathfrak{M}$ does not admit self-loops, individual points are indistinguishable  in $\mathfrak{M}$, and every element has marginal probability $1/3$ to appear in $\mathcal{P}$.
In any finite substructure of $\mathfrak{M}$ with size larger than $2$, we can distinguish the evens and odds, and $\mathfrak{X}$ is clearly exchangeable under preserving the even/odd distinction.  But $\mathfrak{X}$ does not have the representation in Theorem \ref{thm:main1}.  Since $\mathfrak{M}$ is trivial on singletons, the marginal representation of each point must have the form
\[n\in\mathcal{P}\quad\Longleftrightarrow\quad f(\xi_{\emptyset},\xi_{\{n\}})=1,\]
implying that $\mathfrak{X}$ must be fully exchangeable, which it is not.

Notice here that $\mathfrak{M}$ has trivial definable closure but lacks ultrahomogeneity: every even integer can be mapped to any odd integer as singletons, but this embedding cannot be extended to an automorphism of $\mathfrak{M}$.
\end{example}

\section{Relative Exchangeability}\label{section:proofs}

\subsection{Structure of proofs}\label{sketch}
By Proposition \ref{prop:equiv}, we immediately obtain the statement of Theorems \ref{thm:nice1} and \ref{thm:main1} under the weaker condition of Proposition \ref{prop:equiv}(ii).
The observation in Proposition \ref{prop:equiv} adds clarity to our proofs and fosters intuition for the attained representations in our main theorems.
The proofs of Theorems \ref{thm:nice1} and \ref{thm:main1} involve some different technicalities depending on the different assumptions; however, they share a similar structure that we outline here.
The key ideas center on our chosen definition of relative exchangeability and a combination of the Aldous--Hoover theorem (Theorem \ref{AH-general}) and Theorem \ref{AFP_theorem}.

The core of the argument is the same in both cases.
Since $\mathfrak{M}$ has trivial definable closure, Theorem \ref{AFP_theorem} guarantees the existence of an exchangeable, dissociated probability measure $\mu$ such that $\mathfrak{Z}^*\sim\mu$ is isomorphic to $\mathfrak{M}$ with probability 1.
By assumption, $\mathfrak{X}^*$ is distributed according to an $\mathfrak{M}$-exchangeable probability distribution $P_{\mathfrak{M}}$, whose image under relabeling $\mathfrak{X}^*$ by $\sigma$ is an $\mathfrak{M}^{\sigma}$-exchangeable measure $P_{\mathfrak{M}^{\sigma}}$.
Let $\mathfrak{M}^*$ denote the realization of $\mathfrak{Z}^*$, for which we know there exists a permutation $\sigma:|\mathfrak{M}|\to|\mathfrak{M}|$ such that $\mathfrak{M}^{*\sigma}=\mathfrak{M}$.
Given $\mathfrak{M}^*$, we let $\mathfrak{X}^*$ be an $\mathfrak{M}^*$-exchangeable structure from $P_{\mathfrak{M}^*}$ so that the pair $(\mathfrak{M}^*,\mathfrak{X}^*)$ is jointly exchangeable.
We can regard the pair $(\mathfrak{M}^*,\mathfrak{X}^*)$ as a single $\mathcal{L}\cup\mathcal{L}'$-structure, which is exchangeable by construction and, therefore, possesses an Aldous--Hoover representation as in Theorem \ref{AH-general}.

At this point, the details vary based on the additional assumptions about $\mathfrak{M}$, but the main idea is the same.
Since $\mathfrak{M}$ is ultrahomogeneous, then so is $\mathfrak{M}^*$ with probability 1.
Furthermore,  $\age(\mathfrak{M}^*)=\age(\mathfrak{M})$ with probability 1.
In particular, embedded in $\mathfrak{M}^*$ are infinitely many (and in fact a positive density of) copies of every structure in the age of $\mathfrak{M}$.
By ultrahomogeneity, we can go through $\mathfrak{M}^*$ and sequentially choose representatives $\phi(1),\phi(2),\ldots$ such that the domain restriction $\phi\upharpoonright[n]$ is an embedding $\mathfrak{M}|_{[n]}\to\mathfrak{M}^*$ for every $n\geq1$.
Intuitively---we will make this rigorous---the distribution of $\mathfrak{X}^{*\phi}$, given $\mathfrak{M}^*$, depends only on $\mathfrak{M}^{*\phi}=\mathfrak{M}$ and is a copy of an $\mathfrak{M}$-exchangeable structure.

The remainder of the argument relies on a special form of the Aldous--Hoover representation in each case, which in turn determines the nature of our representation for $\mathfrak{X}^*$.
Under $n$-DAP for all $n\geq1$, the distribution of $\mathfrak{Z}^*$ {\em factors through substructures} (Definition \ref{defn:factors}), while under 2-DAP the representation of $\mathfrak{X}^*$ is as in \eqref{eq:random-free}.
Since Aldous--Hoover representations are unique up to measure-preserving transformations, we can always transform to get the appropriate representation.
Theorem \ref{thm:subalgebra} ensures that our choice of embedding $\phi$, which depends on $\mathfrak{Z}^*$ and is therefore random, does not affect the ensuing distribution of $\mathfrak{X}^*$.

We begin with a proof of Theorem \ref{thm:nice1}.

\subsection{Distributions with Enough Amalgamation}\label{section:framewise}

Under the assumption that $\mathfrak{M}$ satisfies $n$-DAP for all $n\geq1$, the following lemma shows that there is a well-behaved representation of $\mathfrak{M}$, a key idea in our proof of Theorem \ref{thm:nice1}.

\begin{definition}\label{defn:factors}
Suppose $\mathfrak{M}$ is the exchangeable structure generated by $f_1,\ldots,f_r$.  We say $\mathfrak{M}$ \emph{factors through substructures} if there are functions $\hat f_i$ so that for almost all $(\xi_s)_{s\subseteq\Nb:\ |s|\leq k},(\prec_s)_{s\subseteq\Nb:\ |s|\leq k}$,
\begin{equation}
f_i((\xi_s)_{s\subseteq\rng x},(\prec_{\vec y})_{\vec y\subseteq\vec x})=\hat f_i(\{\mathfrak{M}|_s\}_{s\subsetneq\rng\vec x},\xi_{\rng\vec x},\prec_{\vec x}).\label{eqn:factor}
\end{equation}
\end{definition}
In general, the variables $\xi_s$ encode the $|s|$-ary information about the structure $\mathfrak{M}$.  When $\mathfrak{M}$ factors through substructures, the only dependence $f_i$ has on the information of arity strictly less than $|\rng\vec x|$ is already realized by the lower arity part of the structure $\mathfrak{M}$.  This means that the functions $f_i$ have no ``hidden'' information: all the information in $\xi_s,\prec_s$ is represented in $\mathfrak{M}|_s$.

\begin{lemma}\label{thm:nice_rep}
  Suppose $\mathfrak{M}$ is ultrahomogeneous and satisfies $n$-DAP for all $n$.  Then there are Borel functions $f_1,\ldots,f_r$ so that the exchangeable structure $\mathfrak{M}^*$ generated by $f_1,\ldots,f_r$ is isomorphic to $\mathfrak{M}$ with probability 1 and factors through substructures.
\end{lemma}
The construction in the following proof is essentially the \emph{frame-wise uniform measure} introduced in \cite{2013arXiv1308.5517B}; see also \cite{KruckmanNote}.
\begin{proof}
For each $n$, let $\age_n(\mathfrak{M})$ be the elements of $\age(\mathfrak{M})$ of size $n$.  Pick any Borel-measurable map $\mathcal{S}:[0,1]\rightarrow\age_1(\mathfrak{M})$ such that for any $\mathfrak{S}\in\age_1(\mathfrak{M})$, $\mathcal{S}^{-1}(\mathfrak{S})$ has positive Lebesgue measure.  Then we set $\hat f_i(\emptyset,\xi_{m},\prec_{m})=1$ if and only if $1\in R^{\mathcal{S}(\xi_m)}_i$.  (Recall that, by our convention, $\mathcal{S}(\xi_m)\in\age_1(\mathfrak{M})$ and therefore $|\mathcal{S}(\xi_m)|=\{1\}$, so we are looking at the unique point of $\mathcal{S}(\xi_m)$.)

Suppose we are given structures $\{\mathfrak{M}^*|_s\}_{s\subsetneq\rng\vec x}$.  Then (after identifying $\rng\vec x$ with $[1,|\rng\vec x|]$) the structures $\mathfrak{S}_i=\mathfrak{M}^*|_{\rng\vec x\setminus\{i\}}$ satisfy the conditions of $n$-DAP.  Let $\mathcal{A}\subseteq\age_n(\mathfrak{M})$ be the set of amalgams and let $\mathcal{A}'$ be a choice of representatives from each isomorphism class of $\mathcal{A}$.  $n$-DAP ensures that $\mathcal{A}$, and therefore $\mathcal{A}'$, is non-empty.  Pick $\mathcal{S}:[0,1]\rightarrow\mathcal{A}'$ Borel-measurable so that for each $\mathfrak{A}\in\mathcal{A}'$, $\mathcal{S}^{-1}(\mathfrak{A})$ has positive measure.
(This is possible because $\mathcal{A}'$ is finite.)

The partial models $\{\mathfrak{M}^*|_s\}_{s\subsetneq\rng\vec x}$ and an amalgam $\mathfrak{A}'\in\mathcal{A}'$ may not be enough to fully specify an amalgam, because $\mathfrak{A}'$ may introduce some new asymmetry---that is, there may be multiple ways to amalgamate $\{\mathfrak{M}^*|_s\}_{s\subsetneq\rng\vec x}$ into an isomorphic copy of $\mathfrak{A}'$.  Since the automorphism group of $\mathfrak{A}'$ is a subgroup of the permutations of $\rng\vec x$, we can associate to each $\prec_{\rng\vec x}$ some such amalgam $\mathfrak{A}_{\prec_{\rng\vec x}}$ so that the association respects the automorphism group of $\mathfrak{A}'$.  Then we can set
$\hat f_i(\{\mathfrak{M}^*|_s\}_{s\subsetneq\rng\vec x},\xi_{\rng\vec x},\prec_{\vec x})=1$ if and only if $\vec x\in R^{\mathfrak{A}_{\prec_{\rng\vec x}}}_i$.

Consider some randomly constructed $\mathfrak{M}^*$ built according to the functions $\hat f_i$.  
By definition $\age(\mathfrak{M}^*)\subseteq\age(\mathfrak{M})$, but also $\age(\mathfrak{M})\subseteq\age(\mathfrak{M}^*)$ with probability 1, as we now show by induction on the size of $|\mathfrak{S}|$.
  It suffices to show that the probability of each $\mathfrak{S}$ occurring is positive.  If $\mathfrak{S}\in\age_1(\mathfrak{M})$ then this is by definition.  If $\mathfrak{S}\in\age_{n+1}(\mathfrak{M})$ and $s=\{s_1,\ldots,s_{n+1}\}\subseteq\mathbb{N}$ with $|s|=n+1$, with positive probability, each $\mathfrak{M}^*|_{\{s_i\}}$ is isomorphic to $\mathfrak{S}|_{\{i\}}$.  Then, since each possible amalgam occurs with positive probability, there is a non-zero chance that each $\mathfrak{M}^*|_{\{s_i,s_j\}}$ is isomorphic to $\mathfrak{S}|_{\{i,j\}}$.  Continuing in this way, there is a non-zero chance that $\mathfrak{M}^*|_s$ is isomorphic to $\mathfrak{S}$.  Therefore, with probability 1, $\mathfrak{S}\in\age(\mathfrak{M}^*)$.

Further, $\mathfrak{M}^*$ is almost surely ultrahomogeneous.  To see this, it suffices to show that for any $S$ with $|S|=n$, any $\phi:S\rightarrow[1,n]$, and any $\mathfrak{T}\in\age_{n+1}(\mathfrak{M}^*)$ so that $\mathfrak{M}^*|_S=\mathfrak{T}^\phi$, there is some $x\not\in S$ so that $\mathfrak{M}^*|_{S\cup\{x\}}=\mathfrak{T}^{\phi_x}$ (where $\phi_x$ extends $\phi$ by $\phi_x(x)=n+1$).  First, since $\mathfrak{T}|_{\{n+1\}}\in\age_1(\mathfrak{M}^*)=\age_1(\mathfrak{M})$, there are infinitely many $x$ so that $\mathfrak{M}^*|_x$ is isomorphic to $\mathfrak{T}|_{\{n+1\}}$.  Since $\mathfrak{T}\in\age(\mathfrak{M}^*)$, $\mathfrak{T}^{\phi_x}$ is one of the possible amalgams of $\{\mathfrak{M}^*|_s\}_{s\subsetneq S\cup\{x\}}$, so for each $x$ there is positive probability that $\mathfrak{M}^*|_{S\cup\{x\}}=\mathfrak{T}^{\phi_x}$.  In particular, with probability $1$, there is some $x$ such that $\mathfrak{M}^*|_{S\cup\{x\}}=\mathfrak{T}^{\phi_x}$.  By a standard back-and-forth argument, $\mathfrak{M}$ and $\mathfrak{M}^*$ are isomorphic.
\end{proof}

\begin{example}
  The natural representation of the random graph is by defining $f(\xi_{\{i,j\}})=1$ if and only if $\xi_{\{i,j\}}\in[0,1/2]$.  (The structure restricted to singletons is trivial and can be ignored.)
\end{example}

\begin{example}
  The random tournament can also be expressed in this way: the structure restricted to a singleton is always trivial, so we define $f(\xi_{\{i,j\}},\prec_{( i,j)})=1$ if and only if $i\prec_{( i,j)}j$.
\end{example}

\begin{example}
  Suppose $\mathcal{L}$ contains a single $k$-ary relation $R$ and we assume $R^{\mathfrak{M}}$ is symmetric and anti-reflexive (only holds for tuples containing $k$ distinct elements) and is non-trivial (contains at least one $k$-tuple but not all $k$-tuples).  Then the representation in Lemma \ref{thm:nice_rep} implies that $\mathfrak{M}$ is the random $k$-ary hypergraph (in particular, all $k$-ary hypergraphs are embedded in $\mathfrak{M}$).  This is because $\mathfrak{M}^*$ restricted to subsets of size less than $k$ is trivial---since there are no relations of arity less than $k$ in $\mathcal{L}$, no substructure of size less than $k$ contains any instances of $R$---so we have $\vec x\in R^{\mathfrak{M}^*}$ if and only if  $f_i(\xi_{\vec x})=1$.  Thus, if $\vec x_1,\ldots,\vec x_d$ are pairwise distinct tuples each consisting of $k$ distinct elements, the events $\vec x_i\in R^{\mathfrak{M}^*}$  and $\{\vec x_j\in R^{\mathfrak{M}^*}\mid j\neq i\}$ are independent, so all $k$-ary hypergraphs appear with positive probability.
\end{example}

We note that the existence of representations which factor through substructures actually characterizes ultrahomogeneous structures with $n$-DAP.

\begin{lemma}
  Suppose $\mathfrak{M}$ is a $\mathcal{L}$-structure on $\mathbb{N}$ and there exists a random $\mathcal{L}$-structure $\mathfrak{M}^*$ which factors through substructures and, with probability $1$, is isomorphic to $\mathfrak{M}$.  Then $\mathfrak{M}$ is ultrahomogeneous and has $n$-DAP for all $n$.
\end{lemma}
\begin{proof}
The argument above shows that $\mathfrak{M}^*$, and therefore $\mathfrak{M}$, is ultrahomogeneous.

To see that $\mathfrak{M}^*$, and therefore $\mathfrak{M}$, has $n$-DAP for all $n$, we proceed inductively.  Consider suitable structures $\{\mathfrak{S}_i\}_{i\leq n}$ in $\age(\mathfrak{M}^*)$, and suppose we have already shown $n-1$-DAP; in particular, we have already shown that there are infinitely many pairwise disjoint sets $S$ with $\phi:S\rightarrow[1,n]$ so that, for each distinct $i,j\in[1,n]$, $\mathfrak{M}^*|_{S\setminus\{\phi^{-1}(i),\phi^{-1}(j)\}}=\mathfrak{S}_i|_{[1,n]\setminus\{i,j\}}^\phi$.

  For each $i$, there is a set $\Xi_i$ of positive measure such that if $\xi_{S\setminus \{\phi^{-1}(i)\}}\in \Xi_i$ and for each $j\in[1,n]\setminus\{i\}$, $\mathfrak{M}^*|_{S\setminus\{\phi^{-1}(i),\phi^{-1}(j)\}}=\mathfrak{S}_i|_{[1,n]\setminus\{i,j\}}^\phi$ then $\mathfrak{M}^*|_{S\setminus\{\phi^{-1}(i)\}}=\mathfrak{S}_i^\phi$.  Since the collection $\{\xi_{S\setminus \{\phi^{-1}(i)\}}\mid i\in[1,n]\}$ is independent, and there are infinitely many choices of $S$, there must be some such set $S$ where each $\xi_{S\setminus \{\phi^{-1}(i)\}}\in\Xi_i$, and therefore $\mathfrak{M}^*|_S$ is an amalgam of $\{\mathfrak{S}_i\}_{i\leq n}$ in $\age(\mathfrak{M}^*)$.
\end{proof}

\subsection{Proof of Theorem \ref{thm:nice1}}\label{sec:proof_nice}

\begin{proof}[Proof of Theorem \ref{thm:nice1}]
We have two languages $\mathcal{L}=\{Q_1,\ldots,Q_r\}$ and $\mathcal{L}'=\{R_1,\ldots,R_{r'}\}$.

  Let $\mathfrak{M}$ and $\mathfrak{X}$ be given.  By Lemma \ref{thm:nice_rep}, we can choose a random exchangeable structure $\mathfrak{M}^*$ that with probability 1 is isomorphic to $\mathfrak{M}$ and $\mathfrak{M}^*$ factors through substructures.  

Since $\mathfrak{M}^*$ is exchangeable and $\mathfrak{X}$ is $\mathfrak{M}^*$-exchangeable, we can combine $\mathfrak{M}^*$ with $\mathfrak{X}$ to obtain an exchangeable probability measure on $\mathcal{L}\cup\mathcal{L}'$-structures $(\mathfrak{Z}^*,\mathfrak{X}^*)$.  By Aldous--Hoover, there exist functions $g_i, h_j$ so that for $(\zeta_s)_{s\subseteq\Nb}$ i.i.d.\ Uniform$[0,1]$ and $(\sqsubset_s)_{s\subseteq\Nb}$ independent uniform random orderings
\begin{itemize}
\item $\vec x\in Q^{\mathfrak{Z}^*}_i\quad\Longleftrightarrow\quad g_i((\zeta_s)_{s\subseteq\rng\vec x},(\sqsubset_{\vec y})_{\vec y\subseteq\vec x})=1$,
\item $\vec x\in R^{\mathfrak{X}^*}_j\quad\Longleftrightarrow\quad h_j((\zeta_s)_{s\subseteq\rng\vec x},(\sqsubset_{\vec y})_{\vec y\subseteq\vec x})=1$.
\end{itemize}

We would like our representation of $\mathfrak{Z}^*$ to factor through substructures; since Aldous--Hoover representations are not quite unique, it takes some tedious work to make this happen, but readers may wish to take this claim for granted and move on to the heart of the proof.
\begin{claim}
Without loss of generality, there are functions $\hat g_i$ so that
\[g_i((\zeta_s)_{s\subseteq\rng\vec x},(\sqsubset_{\vec y})_{\vec y\subseteq\vec x})=\hat g_i(\{\mathfrak{Z}^*\mid_s\}_{s\subsetneq\rng\vec x},\zeta_{\rng\vec x},\sqsubset_{\vec x}).\]
\end{claim}
\begin{claimproof}
Consider the functions $v_1,\ldots,v_r$ generating $\mathfrak{M}^*$ and suppose $\mathfrak{M}^*$ is generated from $v_1,\ldots,v_r$ using the i.i.d.\ Uniform$[0,1]$ random variables $(\xi_s)_{s\subseteq\Nb:\ |s|\leq k}$ and independent uniform random orderings $(\prec_s)_{s\subseteq\Nb:\ |s|\leq k}$.  Because the representation $v_1,\ldots,v_r$ factors through substructures, there are the corresponding functions $\hat v_1,\ldots,\hat v_r$ as in (\ref{eqn:factor}).

Because $\mathfrak{M}^*$ is exchangeable and $\mathfrak{M}^*\equalinlaw\mathfrak{Z}^*$, there is a measure-preserving transformation that takes the Aldous--Hoover representation of $\mathfrak{M}^*$ to that of $\mathfrak{Z}^*$, and vice versa.
In this direction, we take an additional system of variables---$(\xi'_s)$ i.i.d.\ Uniform$[0,1]$ and $(\prec'_s)$ independent uniformly chosen random orderings---and we let $T^d, U^d$ be a family of measure-preserving transformations such that if we set
\[\zeta_s=T^{|s|}((\xi_t)_{t\subseteq s},(\prec_{\vec y})_{\vec y\subseteq\vec x},(\xi'_t)_{t\subseteq s},(\prec'_{\vec y})_{\vec y\subseteq\vec x})\]
and
\[\sqsubset_s=U^{|s|}((\xi_t)_{t\subseteq s},(\prec_{\vec y})_{\vec y\subseteq\vec x},(\xi'_t)_{t\subseteq s},(\prec'_{\vec y})_{\vec y\subseteq\vec x}),\]
we have
\[v_i((\xi_s)_{s\subseteq\rng\vec x},(\prec_{\vec y})_{\vec y\subseteq\vec x})=g_i((\zeta_s)_{s\subseteq\rng \vec x},(\sqsubset_{\vec y})_{\vec y\subseteq\vec x})\]
almost surely.  (That such a measure-preserving transformation exists is a consequence of Kallenberg \cite[Theorem 7.28]{KallenbergSymmetries}.)

By the Coding Lemma \cite[Lemma 2.1]{Aldous1981}, we can encode the pairs of i.i.d.\ Uniform$[0,1]$ variables $(\xi_s,\xi'_s)$ by a single Uniform$[0,1]$ random variable by fixing a measure-preserving function $T':[0,1]\rightarrow[0,1]^2$.  We can further encode the \emph{difference} between $\prec_s$ and $\prec'_s$ by letting $F_s$ be the set of functions from permutations of $[1,|s|]$ to itself.  This gives us a Uniform$[0,1]$ random variable $\xi^\dagger_s$ and a measure-preserving function $V^{|s|}:[0,1]\rightarrow[0,1]\times[0,1]\times F_s$, whose components we write as $V_1^{|s|}$, $V_2^{|s|}$, $V_3^{|s|}$, respectively.
We then set $\xi_s=V_1^{|s|}(\xi^\dagger_s)$, $\xi_s'=V_2^{|s|}(\xi^\dagger_s)$, and $\prec'_s=[V_3^{|s|}(\xi^\dagger_s)](\prec_s)$.  Note that this is ``level preserving'' in the sense that $\xi_s,\xi'_s,\prec_s,\prec'_s$ depends only on $\xi^\dagger_s,\prec_s$ for every $s$.

We then define $g_i^{\dagger}$ and $h_j^{\dagger}$ by
\[ g^\dagger_i((\xi^\dagger_s)_{s\subseteq\rng\vec x},(\prec_{\vec y})_{\vec y\subseteq\vec x})=g_i((\zeta_s)_{s\subseteq\rng \vec x},(\sqsubset_{\vec y})_{\vec y\subseteq\vec x})\]
and
\[h^\dagger_j((\xi^\dagger_s)_{s\subseteq\rng\vec x},(\prec_{\vec y})_{\vec y\subseteq\vec x})=h_j((\zeta_s)_{s\subseteq\rng \vec x},(\sqsubset_{\vec y})_{\vec y\subseteq\vec x}),\]
where $\zeta_s,\sqsubset_s$ are obtained from $(\xi_s^{\dagger})$ and $(\prec_s)$ through the natural compositions of the $T^d,U^d,V^d$. 

By assumption, $\mathfrak{M}$ is ultrahomogeneous and has $n$-DAP for all $n\geq1$ and $\mathfrak{Z}^*$ is exchangeable and isomorphic to $\mathfrak{M}$ with probability 1.
Let $\mathfrak{Z}^{\dagger}$ be the structure generated by the $g^{\dagger}_i$ so that $\mathfrak{Z}^{\dagger}\equalinlaw\mathfrak{Z}^*$.  We can define functions
\[\hat g^\dagger_i(\{\mathfrak{Z}^\dagger|_s\}_{s\subsetneq\rng\vec x},\xi^\dagger_{\rng\vec x},\prec_{\vec x})
=\hat v_i(\{\mathfrak{Z}^\dagger|_s\}_{s\subsetneq\rng\vec x},V_1^{|\vec x|}(\xi^\dagger_{\rng\vec x}),\prec_{\vec x}),\]
and the functions $\hat g^\dagger_i$ show that the representation given by $g^\dagger_i$ factors through substructures as well since
\begin{align*}
  g^\dagger_i((\xi^\dagger_s)_{s\subseteq\rng\vec x},(\prec_{\vec y})_{\vec y\subseteq\vec x})
&=g_i((\zeta_s)_{s\subseteq\rng \vec x},(\sqsubset_{\vec y})_{\vec y\subseteq\vec x})\\
&=v_i((\xi_s)_{s\subseteq\rng\vec x},(\prec_{\vec y})_{\vec y\subseteq\vec x})\\
&=\hat v_i(\{\mathfrak{Z}^\dagger|_s\}_{s\subsetneq\rng\vec x},\xi_{\rng\vec x},\prec_{\vec x})\\
&=\hat g^\dagger_i(\{\mathfrak{Z}^\dagger|_s\}_{s\subsetneq\rng\vec x},\xi^\dagger_{\rng\vec x},\prec_{\vec x}).
\end{align*}
(In passing from the second to third lines above, we once again use the fact that $\mathfrak{Z}^*$ factors through substructures.)
Therefore $\mathfrak{Z}^\dagger$ factors through substructures and we may replace $g_i,h_i$ with $g^\dagger_i,h^\dagger_i$.
\end{claimproof}

For any set $S\subseteq\mathbb{N}$, we define $M(\{\zeta_s\}_{s\subseteq S},(\sqsubset_s)_{s\subseteq S})$ to be the $\mathcal{L}$-structure $\mathfrak{S}$ with $|\mathfrak{S}|=S$ and $\vec x\in\mathcal{Q}^{\mathfrak{S}}_i$ if and only if $g_i((\xi_s)_{s\subseteq\vec x},(\sqsubset_{\vec y})_{\vec y\subseteq\vec x})=1$.  Conversely, given $\mathfrak{S}$ and $(\sqsubset_s)_{s\subseteq\mathfrak{S}}$, we can consider the set
\[\Theta(\mathfrak{S},(\sqsubset_s))=\{\{\zeta_s\}\mid M(\{\zeta_s\}_{s\subseteq S},(\sqsubset_s)_{s\subseteq S})=\mathfrak{S}\}.\]
We have chosen $\mathfrak{M}^*$ so that it factors through substructures, implying that $\Theta(\mathfrak{S},(\sqsubset_s)_{s\subseteq S})$ is a cube, i.e., for suitable functions $\Theta_s(\mathfrak{S},\sqsubset_s)$ it has the form $\prod_{s\subseteq S}\Theta_s(\mathfrak{S}|_s,\sqsubset_s)$ .  This is because whether or not $M(\{\zeta_s\}_{s\subseteq S},(\sqsubset_s)_{s\subseteq S})=\mathfrak{S}$ depends only on $M(\{\zeta_t\}_{t\subseteq s},(\sqsubset_t)_{t\subseteq s})$ for $s\subsetneq S$ and the values $\zeta_S,\prec_S$.

In particular, if $\mathfrak{S}\in\age(\mathfrak{M})$, then there is a conditional measure on $\Theta_s(\mathfrak{S}|_s,\sqsubset_s)$, and we may define a measure-preserving function $\theta_s^{\mathfrak{S},(\sqsubset_t)_{t\subseteq s}}:[0,1]\rightarrow\Theta_s(\mathfrak{S}|_s,\sqsubset_s)$.  We may now define
\[f_j(\mathfrak{S},(\xi_s)_{s\subseteq\rng \vec x},(\sqsubseteq_{\vec y})_{\vec y\subseteq\vec x})=h_j((\theta_s^{\mathfrak{S}|_s,(\sqsubseteq_t)_{t\subseteq s}}(\xi_s))_{s\subseteq\rng \vec x},(\sqsubseteq_{\vec y})_{\vec y\subseteq\vec x}).\]

We now return to the original structure $\mathfrak{M}$.  Let $\mathfrak{X}^{**}$ be the structure generated by the $f_j$ using $\mathfrak{M}$; that is,
\[\vec x\in R_j^{\mathfrak{X}^{**}}\quad\Longleftrightarrow\quad f_j(\mathfrak{M}|_{\rng\vec s},(\xi_s)_{s\subseteq\rng \vec x},(\sqsubseteq_{\vec y})_{\vec y\subseteq\vec x}).\]
Observe that for any $S\subseteq\mathbb{N}$, $\mathbb{P}(\mathfrak{X}^{**}|_S=\mathfrak{T})$ is equal to $\mathbb{P}(\mathfrak{X}^*|_S=\mathfrak{T} \mid \mathfrak{Z}^*|_S=\mathfrak{M}|_S)$ (that is, the conditional probability that $\mathfrak{X}^*|_S=\mathfrak{T}$, given that $\mathfrak{Z}^*|_S=\mathfrak{M}|_S$).
Recall that $\mathfrak{X}$ is $\mathfrak{M}$-exchangeable and $\mathfrak{X}^*$ is chosen to be relatively exchangeable with respect to $\mathfrak{Z}^*$ (and $\mathfrak{Z}^*$ is exchangeable and isomorphic to $\mathfrak{M}$ with probability 1).
By our choice of $\mathfrak{X}^*$, we have
\[\mathbb{P}(\mathfrak{X}^*|_S=\mathfrak{T}\mid \mathfrak{Z}^*|_S=\mathfrak{M}|_S);\]
whence, $\mathfrak{X}\equalinlaw\mathfrak{X}^{**}$.
\end{proof}

\begin{proof}[Proof of Lemma \ref{thm:nice_dissociated}]
Suppose that $\mathfrak{X}$ is dissociated.  Then, from the previous theorem, we have the structure $\mathfrak{X}^{**}\equalinlaw\mathfrak{X}$ generated by the $f_j$.  For any $\xi_\emptyset$, define
\[f_j^{\xi_\emptyset}(\mathfrak{S},(\xi_s)_{\emptyset\subsetneq s\subseteq\rng\vec x},(\sqsubseteq_{\vec y})_{\vec y\subseteq\vec x})=
f_j(\mathfrak{S},(\xi_s)_{s\subseteq\rng \vec x},(\sqsubseteq_{\vec y})_{\vec y\subseteq\vec x}).\]
For each $\xi_\emptyset$, the functions $f_j^{\xi_\emptyset}$ generate a structure $\mathfrak{X}^{\xi_\emptyset}$.  We claim that for almost every $\xi_\emptyset$, $\mathfrak{X}^{\xi_\emptyset}\equalinlaw\mathfrak{X}^{**}$.  It suffices to show that, for each $S$ and each $\mathfrak{S}$, for almost every $\xi_\emptyset$, $\mathbb{P}(\mathfrak{X}^{\xi_\emptyset}|_S=\mathfrak{S})=\mathbb{P}(\mathfrak{X}^{**}|_S=\mathfrak{S})$.

Toward a contradiction, suppose that for some $S$ and $\mathfrak{S}$, there are positive measure of $\xi_\emptyset$ such that $\mathbb{P}(\mathfrak{X}^{\xi_\emptyset}|_S=\mathfrak{S})\neq \mathbb{P}(\mathfrak{X}^{**}|_S=\mathfrak{S})$.  
Using the ultrahomogeneity of $\mathfrak{M}$, we can find $T$ with $S\cap T=\emptyset$ so that $\mathfrak{M}|_S$ is isomorphic to $\mathfrak{M}|_T$ and, therefore, so $\mathbb{P}(\mathfrak{X}^{\xi_\emptyset}|_T=\mathfrak{S})=\mathbb{P}(\mathfrak{X}^{\xi_\emptyset}|_S=\mathfrak{S})$ and $\mathbb{P}(\mathfrak{X}^{**}|_T=\mathfrak{S})=\mathbb{P}(\mathfrak{X}^{**}|_S=\mathfrak{S})$.  
However, by the shared dependence of $\mathfrak{X}^{\xi_0}|_S$ and $\mathfrak{X}^{\xi_0}|_T$ on $\xi_0$, the events $\{\mathfrak{X}^{**}|_T=\mathfrak{S}\}$ and $\{\mathfrak{X}^{**}|_S=\mathfrak{S}\}$ are positively correlated, implying that $\mathbb{P}(\mathfrak{X}^{**}|_T=\mathfrak{S}\mid \mathfrak{X}^{**}|_S=\mathfrak{S})>\mathbb{P}(\mathfrak{X}^{**}|_T=\mathfrak{S})$.
But this contradicts the dissociation of $\mathfrak{X}^{**}$.

\end{proof}

\subsection{Sufficiently Large Product Algebras}

In this subsection we give a technical result showing that a Borel function $f:[0,1]^k\rightarrow[0,1]$ is measurable with respect to a smaller $\sigma$-algebra $\mathcal{B}^k$ where the only sets in $\mathcal{B}$ are ones which can be defined from $f$ in a certain way.

\begin{definition}
  Let $\{f_j\}$ be a countable collection of functions on $[0,1]^{k_j}$.  We say $\tilde v:[0,1]^d\rightarrow [-1,1]$ is \emph{generated by the $f_j$} if there exists a function $v$, values $j_1,\ldots,j_r$, and tuples $\vec c^i$ for $i\leq r$ such that
\[\tilde v(\zeta_1,\ldots,\zeta_d)=v(f_{j_1}(\zeta_{c^1_1},\ldots,\zeta_{c^1_{k_{j_1}}}),\ldots,f_{j_r}(\zeta_{c^r_1},\ldots,\zeta_{c^r_{k_{j_r}}}))\]
\end{definition}

\begin{theorem}\label{thm:subalgebra}
Let $\{f_j\}$ be a countable collection of Borel-measurable functions on $[0,1]^{k_j}$.  Suppose we have fixed a measure $\mu$ on $\mathcal{B}$.  Then there is a $\sigma$-algebra $\mathcal{B}$ such that:
  \begin{itemize}
  \item Each $f_j$ is measurable with respect to $\mathcal{B}^{k_j}$,
  \item $\mathcal{B}$ is generated by sets of the form $\{\zeta\mid \tilde v(\zeta,\zeta_2,\ldots,\zeta_d)\in I\}$, where $I$ is an interval and $\tilde v$ is generated by the $f_j$.
  \end{itemize}
Furthermore,  if for each $d$ we have a set $B_{d-1}\subseteq\mathcal{B}^{d-1}$ with $\mu(B_{d-1})=0$, we may choose the generating sets $\{\zeta\mid \tilde v(\zeta,\zeta_2,\ldots,\zeta_d)\in I\}$ so that $(\zeta_2,\ldots,\zeta_d)\not\in B_{d-1}$.
\end{theorem}

\begin{proof}

  It clearly suffices to show this when the collection of function $\{f_j\}$ consists of a single function $f$ on $[0,1]^k$, since if there are multiple functions, we can simply take the union of the corresponding $\sigma$-algebras.  Without loss of generality, we assume the sets $B_d$ are closed under permutations and that for any $i$, if $\vec\zeta\not\in B_d$ then the set of $\vec\zeta'$ such that $(\vec\zeta,\vec\zeta')\in B_{d+i}$ has measure $0$.

If $k=1$ this is trivial, so assume $k>1$.

We say a set is \emph{built from $f$} if it has the form $\{\zeta\mid \tilde v(\zeta,\zeta_2,\ldots,\zeta_d)\in I\}$ with $\tilde v$ generated by $f$, $I$ an interval, and $(\zeta_2,\ldots,\zeta_d)\not\in B_{d-1}$.  We build $\mathcal{B}$ in countably many stages, beginning with the trivial $\sigma$-algebra $\mathcal{B}_0$, with each stage finitely generated by sets built from $f$.

Suppose we have a $\sigma$-algebra $\mathcal{B}$ generated by finitely many sets built from $f$.  We call $\mathcal{B}'$ a \emph{good extension} of $\mathcal{B}$ if:
\begin{itemize}
\item $\mathcal{B}\subseteq\mathcal{B}'$,
\item $\mathcal{B}'$ is generated by $\mathcal{B}\cup\{B_1,\ldots,B_{k'}\}$, where $k'\leq k$ and each $B_i$ is built from $f$,
\item $||\mathbb{E}(f\mid(\mathcal{B}')^k)||_{L^2}>||\mathbb{E}(f\mid\mathcal{B}^k)||_{L^2}$.
\end{itemize}

We claim that if $f$ is not measurable with respect to $\mathcal{B}^k$ then a good extension exists.  Given $\mathcal{B}$ where $f$ is not measurable with respect to $\mathcal{B}^k$, let $f'=f-\mathbb{E}(f\mid\mathcal{B}^k)$
so
\[0<\int [f'(\zeta_1,\ldots,\zeta_k)]^2 d\mu^k.\]

$\mathbb{E}(f\mid\mathcal{B})$ has the form $\sum_i\lambda_i\chi_{\prod_{j\leq k}B_{i,j}}$, where the $\prod_{j\leq k}B_{i,j}$ are rectangles from $\mathcal{B}^k$.  Since each $B_{i,j}$ is a finite union of finite intersections of sets built from $f$, we may expand all these unions and intersections and, without loss of generality, assume that $B_{i,j}$ itself is built from $f$; and since $B_{i,j}=\{\zeta\mid\tilde \nu_{i,j}(\zeta,\vec\zeta_{i,j})\in I\}$ for some $\tilde \nu_{i,j}$, we may define $\tilde \nu'_{i,j}(\zeta,\vec\zeta_{i,j})$ to be the characteristic function of this set.  So $f'(\zeta_1,\ldots,\zeta_k)$ has the form
\[f(\zeta_1,\ldots,\zeta_k)-\sum_i\lambda_i\prod_{j\leq k}\tilde \nu'_{i,j}(\zeta_i,\vec \zeta_{i,j}).\]

By Cauchy--Schwarz, we have
\begin{align*}
0
&<\int [f'(\zeta_1,\ldots,\zeta_k)]^2 d\mu^k\\
&=\int \int [f'(\zeta_1,\ldots,\zeta_k)]^2 d\mu^{k-1}d\mu(\zeta_1)\\
&\leq \sqrt{\int \left(\int f'(\zeta_1,\zeta_2\ldots,\zeta_k)d\mu^{k-1}\right)^2 d\mu(\zeta_1)}\\
&=\sqrt{\int \int f'(\zeta_1,\zeta_2^0\ldots,\zeta^0_k)f'(\zeta_1,\zeta_2^1,\ldots,\zeta^1_k) d\mu^{2k-1}(\vec{\zeta}^{0},\vec{\zeta}^1)d\mu(\zeta_1) }.\\
\end{align*}
Iterating this process for each coordinate $i<k$ and raising to the $2^k$, we have 
\[0<\int \prod_{\tau:[1,k]\rightarrow\{0,1\}} f'(\zeta_1^{\tau\upharpoonright ([1,k]\setminus\{1\})},\ldots,\zeta^{\tau\upharpoonright ([1,k]\setminus\{k\})}_k) d\mu^{k2^{k-1}}.\]
In this integral, for each $i\in[1,k]$, we have a copy of $\zeta_i$ for each function $\tau:([1,k]\setminus\{i\})\rightarrow \{0,1\}$.  Observe that if $\tau\neq\tau'$ in the product, there is at most one $i$ such that $\zeta_i^{\tau\upharpoonright ([1,k]\setminus\{i\})}=\zeta_i^{\tau'\upharpoonright ([1,k]\setminus\{i\})}$---if there is any such $i$ then $\tau(j)=\tau'(j)$ for $j\neq i$; if we also had $\tau(i)=\tau'(i)$ then we would have $\tau=\tau'$.  It is also easy to see that each $\zeta^\sigma_i$ appears exactly twice in the product.


The important feature of this product is that each term has the form
\[f'(\zeta_1,\ldots,\zeta_k),\]
where each $\zeta_i$ is chosen from one of $2^{k-1}$ copies.  We have a distinguished choice $\zeta^{\vec 0}_i$ for each $i$: one element of our product is $f'(\zeta^{\vec 0}_1,\ldots,\zeta^{\vec 0}_k)$ and any other copy of $f'$ includes at most one $\zeta^{\vec 0}_i$ in its list of inputs.

Therefore, we can rewrite this product
\[0<\iint f'(\zeta^{\vec 0}_1,\ldots,\zeta^{\vec 0}_k) \prod_i f'(\zeta^{\vec 0}_i,\vec\zeta^X_i) g(\vec\zeta^X) d\mu^k d\mu^{k2^{k-1}-k}(\vec\zeta^X),\]
separating all the other variables into $\vec\zeta^X$.  In particular, there is a set of $\vec\zeta^X$ of positive measure such that
\[0<|\int f'(\zeta^{\vec 0}_1,\ldots,\zeta^{\vec 0}_k) \prod_i f'(\zeta^{\vec 0}_i,\vec\zeta^X_i) d\mu^k|.\]

When we expand out $f'$ in the product $\prod_i f'(\zeta^{\vec 0}_i,\vec\zeta^X_i)$, we get a large sum of products of the form
\[\prod_{i} \tilde \nu^*_{i}(\zeta_i,\vec\zeta^X_i,\vec\zeta^Y_i)\]
where the $\vec\zeta^X_i$ as in the previous equation and the $\vec\zeta^Y_i$ are the fixed parameters appearing in the construction of the sets in $\mathcal{B}$.  The level sets of this sum can be approximated by unions of sets of the form $\prod_i B_i$ where each $B_i$ has the form $\{\zeta_i\mid \nu^*_i(\zeta_i,\vec\zeta^X_i,\vec\zeta^Y_i)\}$.  Therefore there must be some sets $B_i$ of this form so that the set of $\vec\zeta^X_i$ making $|\int_{\prod_i B_i} f'd\mu^k|>0$ has positive measure.  Therefore we can choose parameters $\vec\zeta^X_i$ so that $(\vec\zeta^X_i,\vec\zeta^Y_i)\not\in B_d$.  Taking $\mathcal{B}'$ to be the $\sigma$-algebra generated by $\mathcal{B}\cup\{B_1,\ldots,B_k\}$, $||\mathbb{E}(f\mid\mathcal{B}')||_{L^2}>||\mathbb{E}(f\mid\mathcal{B})||_{L^2}$.  This shows the existence of good extensions.

Let $\mathcal{B}_0$ be the trivial $\sigma$-algebra.  Given $\mathcal{B}_i$, if $f$ is not measurable with respect to $\mathcal{B}_i^k$, we choose $\mathcal{B}_{i+1}$ among all good extensions of $\mathcal{B}$ so that whenever $\mathcal{B}'$ is a good extension of $\mathcal{B}_i$,
\[||\mathbb{E}(f\mid\mathcal{B}')||_{L^2}-||\mathbb{E}(f\mid\mathcal{B}_{i+1})||_{L^2}<2(||\mathbb{E}(f\mid\mathcal{B}_{i+1})||_{L^2}-||\mathbb{E}(f\mid\mathcal{B}_i)||_{L^2}).\]
(In other words, $\mathcal{B}_{i+1}$ contains at least half as much information as any other good extension.)  Let $\epsilon_i=||\mathbb{E}(f\mid\mathcal{B}_{i+1})||_{L^2}-||\mathbb{E}(f\mid\mathcal{B}_i)||_{L^2}$.

We let $\mathcal{B}=\bigcup_i\mathcal{B}_{i+1}$.  Observe that $||f||_{L^2}\geq\sum_i\epsilon_i$, so $\epsilon_i\rightarrow 0$.  In particular, if $f$ were not measurable with respect to $\mathcal{B}$, we could find a good extension $\mathcal{B}'\supsetneq\mathcal{B}$ with $||\mathbb{E}(f\mid\mathcal{B}')||_{L^2}\geq||\mathbb{E}(f\mid\mathcal{B})||_{L^2}+\epsilon$.  But for some $i$, $\epsilon_i<\epsilon/2$, contradicting the choice of $\mathcal{B}_{i+1}$.
\end{proof}

\subsection{Proof of Theorem \ref{thm:main1}}\label{sec:proof_main}

\begin{theorem}\label{thm:canonical rep}
Let $\mathcal{L},\mathcal{L}'$ be signatures and $\mathfrak{M}$ be an ultrahomogeneous $\mathcal{L}$-structure whose age has DAP.
Suppose $\mathfrak{X}=(\Nb,\mathcal{X}_1,\ldots,\mathcal{X}_{r'})$ is a dissociated $\mathfrak{M}$-exchangeable $\mathcal{L}'$-structure.
Then there are age compatible Borel functions $f_1,\ldots,f_{r'}$ so that the age indexed $\mathcal{L}'$-structure $\{\mathfrak{Y}^{\mathfrak{S}}\}$ generated by $f_1,\ldots,f_{r'}$ satisfies $\mathfrak{Y}^{\mathfrak{S}}\equalinlaw\mathfrak{X}^{\mathfrak{S}}$ for all $ {\mathfrak{S}}\in\age(\mathfrak{M})$.
\end{theorem}

Before giving the actual proof, we give a brief outline.  As in the proof of Theorem \ref{thm:nice1}, we will replace $\mathfrak{M}$ with an exchangeable representation and then apply Aldous--Hoover to obtain a representation $(\mathfrak{Z}^*,\mathfrak{X}^*)$ of the combined structure.  After a measure-preserving transformation, we will ensure that the representation of $\mathfrak{M}$ is ``random-free''---that is, depends only on the singleton data $\zeta_i$.

We will then apply Theorem \ref{thm:subalgebra} to decompose the singleton random variables $\zeta_i$ into two independent random variables $\eta_i$ and $\xi_i$ so that $\eta_i$ captures all the information needed to construct $\mathfrak{Z}^*$, while $\xi_i$ represents the remaining information in $\zeta_i$ which is needed to construct $\mathfrak{X}^*$.

The representation we construct will depend on the random data $\xi_i$.  Given the $\xi_i$ and a $\mathfrak{S}\in\age(\mathfrak{M})$, we use $\mathfrak{S}$ to ``guess'' what the values $\eta_i$ might have been: specifically, we choose a ``typical'' sequence of values $\eta_i$ which would have caused $\mathfrak{Z}^*|_S=\mathfrak{S}$.  Given $\eta_i$ and $\xi_i$, we can reconstruct the value $\zeta_i$, which is the data needed to construct $\mathfrak{X}^*|_S$.

\begin{proof}
By Ackerman--Freer--Patel \cite{AFP}, there is an exchangeable probability measure $\mu$ concentrated on $\mathcal{L}$-structures $\mathfrak{M}'$ isomorphic to $\mathfrak{M}$ given in the form
\[\vec x\in R^{\mathfrak{M}'}_i\quad\Longleftrightarrow\quad v_i((\xi_j)_{j\in\rng\vec x}).\]
By assumption, there is an $\mathfrak{M}$-exchangeable measure $\mu'$ such that the restriction of $\mathfrak{X}'\sim\mu'$ to $S$ depends only on $\mathfrak{M}|_S$, for every $S\subseteq\mathbb{N}$.

As sketched in Section \ref{sketch}, we can put the exchangeable structure $\mathfrak{M}'$ and an $\mathfrak{M}'$-exchangeable structure $\mathfrak{X}'$ together to obtain an exchangeable probability measure on $\mathcal{L}\cup\mathcal{L}'$-structures $(\mathfrak{Z}^*,\mathfrak{X}^*)$.  
We write $\mathfrak{M}^*$ for the $\mathcal{L}$-structure corresponding to $\mathfrak{Z}^*$, which is isomorphic to $\mathfrak{M}$ with probability 1.  By Aldous--Hoover, there exist functions $g_i, h_j$ and a collection $(\zeta_{s})_{s\subseteq\Nb:\ |s|\leq k}$ of i.i.d.\ Uniform$[0,1]$ random variables and $(\prec_s)_{s\subseteq\Nb:\ |s|\leq k}$ uniform random orderings so that
\begin{itemize}
\item $\vec x\in Q_i^{\mathfrak{Z}^*}\quad\Longleftrightarrow\quad g_i((\zeta_s)_{s\subseteq\rng\vec x,|s|>0},(\prec_{\vec y})_{\vec y\subseteq\vec x})$,
\item $\vec x\in R_j^{\mathfrak{X}^*}\quad\Longleftrightarrow\quad h_{j}((\zeta_{ s})_{ s\subseteq\rng\vec x,|s|>0},(\prec_{\vec y})_{\vec y\subseteq\vec x})$, and
\item for any $V\subseteq\mathbb{N}$, the distribution of $\mathfrak{X}^*|_{V}$ given that $\mathfrak{M}^*|V={\mathfrak{S}}$ is the same as the distribution of $\mathfrak{X}^{ {\mathfrak{S}}}$.
\end{itemize}
As in the proof of Theorem \ref{thm:nice1}, we may apply a measure-preserving transformation so that our representation of $\mathfrak{Z}^*$ has the same form as $\mu$---that is, depends only on singletons---so we may assume without loss of generality that
\[\vec x\in Q_i^{\mathfrak{Z}^*}\quad\Longleftrightarrow\quad g_i((\zeta_j)_{j\in\rng\vec x})=1.\]

For any $S\subseteq\mathbb{N}$, we define $M(\{\zeta_n\}_{n\in S})$ to be the $\mathcal{L}$-structure $\mathfrak{S}$ with $|\mathfrak{S}|=S$ and $\vec x\in\mathcal{Q}_i^{\mathfrak{S}}$ if and only if $g_i((\zeta_j)_{i\in\rng\vec x})=1$.   We intend to use Theorem \ref{thm:subalgebra} to choose a $\sigma$-algebra $\mathcal{B}$ so that each $g_i$ is measurable with respect to $\mathcal{B}^{k_i}$ and $\mathfrak{X}^*|_S$ is independent of $\mathcal{B}$ after conditioning on $M(\{\zeta_n\}_{n\in S})$.

We write $Z_{s,\mathfrak{S}}$ for the event $\{M(\{\zeta_n\}_{n\in s})=\mathfrak{S}\}$ and $X_{s,\mathfrak{T}}$ for the event $\{\mathfrak{X}^*|_s=\mathfrak{T}\}$.

\begin{claim}
Let $s\subseteq\mathbb{N}$ by finite.
For each $n\in s$, let $\tilde v_n$ be generated by the $f_i$, let a rectangle $I$ in $[0,1]^{|s|}$ be given, and let
\[V_{\{\tilde v_n\},I}(\{\vec\zeta^n\}_{n\in S})=\{\{\zeta_n\}_{n\in s}\mid \prod_{n\in s}\tilde v_n(\zeta_n,\vec\zeta^n)\in I\}.\]

Then the set of $\vec\zeta^n$ such that
\[\mathbb{P}(X_{s,\mathfrak{T}}\mid V_{\{\tilde v_n\},I}(\{\vec\zeta^n\}_{n\in s})\text{ and }Z_{s,\mathfrak{S}})\neq\mathbb{P}(T_{s,\mathfrak{T}}\mid Z_{s,\mathfrak{S}})\]
has measure $0$.
\end{claim}
\begin{claimproof}
The set $\bigcup_{n\in s}\{\zeta_n,\vec\zeta^n\}$ is a finite set of random variables.
We fix an enumeration $t\supseteq s$ of these variables so that if $n\neq n'$ then $\vec\zeta^n$ and $\vec\zeta^{n'}$ are mutually independent sets of random variables. 
Let $V_{\{\tilde v_n\},I}=\bigcup_{\{\vec\zeta^n\}}V_{\{\tilde v_n\},I}(\{\vec\zeta^n\})$ be the event on $\{\zeta_n\}_{n\in t}$ that there is some $\{\vec\zeta^n\}$ so that $\prod_{n\in t}\tilde v_n(\zeta_n,\vec\zeta_n)\in I$, i.e., $V_{\{\tilde v_n\},I}(\{\vec\zeta^n\})$ holds.  The event $V_{\{\tilde v_n\},I}\cap Z_{s,\mathfrak{S}}$ is determined by $M(\{\zeta_n\}_{n\in t})$ and implies that $M(\{\zeta_n\}_{n\in t})|_s\mathop{=}\mathfrak{S}$.  Therefore, $V_{\{\tilde v_n\},I}\cap Z_{s,\mathfrak{S}}$ is a union of events of the form $Z_{t,\mathfrak{S}'}$ where $\mathfrak{S}'|_s=\mathfrak{S}$.  But for any such $\mathfrak{S}'$ and any $\mathcal{L}'$-structure $\mathfrak{T}'$ with $|\mathfrak{T}'|=t$, we have
\[\mathbb{P}(X_{t,\mathfrak{T}'}\mid Z_{t,\mathfrak{S}'})=\mathbb{P}(\mathfrak{X}^{\mathfrak{S}'}=\mathfrak{T}'),\]
and, since $\mathfrak{X}^*|_s$ depends only on $\mathfrak{M}^*|_s$, 
\[\mathbb{P}(X_{s,\mathfrak{T}}\mid Z_{t,\mathfrak{S}'})=\mathbb{P}(\mathfrak{X}^{\mathfrak{S}}=\mathfrak{T}).\]

Therefore
\begin{align*}
  \mathbb{P}(X_{s,\mathfrak{T}}\mid  V_{\{\tilde v_n\},I}(\{\vec\zeta^n\})\text{ and }Z_{s,\mathfrak{S}})
&=\sum_{\mathfrak{S}':\ \mathfrak{S}'|_s=\mathfrak{S}}\mathbb{P}(X_{s,\mathfrak{T}}\mid Z_{t,\mathfrak{S}'})\mathbb{P}(Z_{t,\mathfrak{S}'}\mid V_{\{\tilde v_n\},I}(\{\vec\zeta^n\})\text{ and }Z_{s,\mathfrak{S}})\\
&=\mathbb{P}(X_{s,\mathfrak{T}}\mid Z_{s,\mathfrak{S}})\sum_{\mathfrak{S}':\ \mathfrak{S}'|_s=\mathfrak{S}}\mathbb{P}(Z_{t,\mathfrak{S}'}\mid V_{\{\tilde v_n\},I}(\{\vec\zeta^n\})\text{ and }Z_{s,\mathfrak{S}})\\
&=\mathbb{P}(X_{s,\mathfrak{T}}\mid Z_{s,\mathfrak{S}}).
\end{align*}

\end{claimproof}

Choose a countable collection of functions $\tilde v$ dense in the collection of functions generated by the $f_i$ and let $B_d$ be the union of all countably many sets of measure $0$ given by the previous claim over all $\tilde v$, $s$, $\mathfrak{S},\mathfrak{T}$.

By Theorem \ref{thm:subalgebra}, we can choose a $\sigma$-algebra $\mathcal{B}$ so that
\begin{itemize}
\item each $f_i$ is measurable with respect to $\mathcal{B}^{k_i}$ and
\item each set in $\mathcal{B}^{|\mathfrak{S}|}$ is generated by sets 
\[V_{\{\tilde v_n\}}(\{\vec\zeta^n\})=\{\{\zeta_n\}_{n\in |{\mathfrak{S}}|}\mid\prod_{n\in |{\mathfrak{S}}|}\tilde v_n(\zeta_n,\vec\zeta^n)\in I\},\]
so that any $X_{s,\mathfrak{T}}$ is independent of $\mathcal{B}$ when conditioned on $ M(\{\zeta_{n}\}_{n\in |{\mathfrak{S}}|})={\mathfrak{S}}$.
\end{itemize}

We decompose $\zeta_i=h(\eta_i,\xi_i)$, where $\mathcal{B}$ is measurable with respect to the $\eta_i$ component alone.  For $|\vec y|>1$, we set $\xi_{\vec y}=\zeta_{\vec y}$.

Now, for each $\mathfrak{S}$ with $|\mathfrak{S}|=[1,n]$, we wish to choose a single value $\eta_{{\mathfrak{S}}}$ (depending on the values chosen for $\eta_{{\mathfrak{S}}\upharpoonright[1,n']}$ for $n'<n$) so that setting
\[f_{j}({\mathfrak{S}}|_{[\max\vec x]},(\xi_{ s})_{ s\subseteq\rng\vec x},(\prec_{\vec y})_{\vec y\subseteq \vec x})=f'_{j}((g(\eta_{{\mathfrak{S}}\upharpoonright i},\xi_i))_{i\in \rng\vec x},(\xi_s)_{s\subseteq\rng\vec x,|s|>1},(\prec_{\vec y})_{\vec y\subseteq \vec x})\]
satisfies the theorem.

It suffices to show that for any finite $\mathfrak{S}$, $\mathfrak{T}$, the set of $\eta_{ {\mathfrak{S}}\upharpoonright 1},\ldots,\eta_{ {\mathfrak{S}}\upharpoonright n}=\eta_{{\mathfrak{S}}}$ such that $M(\{\eta_{ {\mathfrak{S}}\upharpoonright i}\})\mathop{=}{\mathfrak{S}}$ and
\[\mathbb{P}(Z_{s,\mathfrak{T}})\neq \mathbb{P}(\mathfrak{X}^{\mathfrak{S}}=\mathfrak{T})\]
(where the first probability is over choices of $\xi_{\vec y}$) has measure $0$.  For then we choose the sequence $\eta_{{\mathfrak{S}}\upharpoonright 1},\ldots,\eta_{{\mathfrak{S}}\upharpoonright n}$ successively, avoiding the set of measure $0$ of choices for $\eta_{{\mathfrak{S}}\upharpoonright i}$ which either belong to such a set, or which cause the set of extensions belonging to such a set to have positive measure.

Towards a contradiction, suppose that for some finite $ {\mathfrak{S}}$, $\mathfrak{T}$, the set of $\eta_{ {\mathfrak{S}}\upharpoonright 1},\ldots,\eta_{ {\mathfrak{S}}\upharpoonright n}$ such that $ M(\{\eta_{ {\mathfrak{S}}\upharpoonright i}\})\mathop{=}{\mathfrak{S}}$ and
\[\mathbb{P}(Z_{s,\mathfrak{T}})\neq \mathbb{P}(X^{\mathfrak{S}}=\mathfrak{T})\]
 has positive measure.  Then there exists some set $B$ in $\mathcal{B}^{|{\mathfrak{S}}|}$ so that $ M(\{\zeta_{\phi(n)}\})^\phi\mathop{=}{\mathfrak{S}}$ for all $\{\zeta_{\phi(n)}\}\in\mathcal{B}^{|{\mathfrak{S}}|}$ but $\mathbb{P}(Z_{s,\mathfrak{T}})\neq \mathbb{P}(X^{\mathfrak{S}}=\mathfrak{T})$.  This contradicts the construction of $\mathcal{B}$.

It follows that we may choose $\eta_{ {\mathfrak{S}}}$ by induction on $| {\mathfrak{S}}|$ and set
\[f_{j}( {\mathfrak{S}}|_{[\max\vec x]},(\xi_{ s})_{ s\subseteq\rng\vec x},(\prec_{\vec y})_{\vec y\subseteq \vec x})=f'_{j}((g(\eta_{ {\mathfrak{S}}\upharpoonright i},\xi_i))_{i\in \rng\vec s},(\xi_{ s})_{s\subseteq\rng\vec x,|s|>1},(\prec_{\vec y})_{\vec y\subseteq\vec x}).\]
\end{proof}

\begin{proof}[Proof of Theorem \ref{thm:main1}]
Follows immediately from Theorem \ref{thm:canonical rep}.
\end{proof}

\section{Concluding remarks}

\subsection{Applications to Markov chains}
A natural setting for relatively exchangeable structures is in the study of Markov chains on combinatorial state spaces.
A  {\em (discrete-time) Markov chain} on $\mathcal{X}_{\Nb}\subseteq\lN$ is a collection $\mathbf{X}=(X_t)_{t=0,1,\ldots}$ of random $\mathcal{L}$-structures whose distribution is determined by an {\em initial distribution} $\mu$ and a family of {\em transition probabilities}
\[P(x,\cdot):=\mathbb{P}\{X_{t+1}\in\cdot\mid (X_s)_{0\leq s\leq t},\ X_t=x\},\quad x\in\mathcal{X}_{\Nb},\quad\text{for all }t\geq0.\]
Given those ingredients, $\mathbf{X}$ is generated by $X_0\sim\mu$ and, given $X_t=x$, $X_{t+1}\sim P(x,\cdot)$ for every $t\geq0$.

Special cases of these processes, e.g., partition-valued processes \cite{Bertoin2006,Crane2014AOP,Pitman2005} and graph-valued processes \cite{Crane2014GraphsI,Crane2014GraphsII},  arise  in various statistical applications, where the assumptions of exchangeability and consistency are natural.
We call $\mathbf{X}$ {\em exchangeable} if $\mathbf{X}^{\sigma}:=(X_t^{\sigma})_{t=0,1,\ldots}$ and $\mathbf{X}$ are versions of the same Markov chain for every permutation $\sigma:\Nb\to\Nb$; and $\mathbf{X}$ is {\em consistent} if the restriction $\mathbf{X}|_{[n]}:=(X_t|_{[n]})_{t=0,1,\ldots}$ to $\mathcal{L}$-structures with domain $[n]$ is also a Markov chain for every $n\geq1$.
Relative exchangeability arises naturally in this context: for every $t\geq0$, the exchangeability and consistency properties imply that $X_{t+1}$ is relatively exchangeable with respect to $X_t$.
The consistency assumption is, in fact, stronger than relative exchangeability because it must account for variability in the reference structures $X_t$ for every $t\geq0$: the transition probabilities of $\mathbf{X}$ entail an ensemble of relatively exchangeable structures that fit together in an appropriate way.
We consider these and other relevant questions about combinatorial Markov processes in the companion article \cite{CraneTowsner2015b}.

\subsection{Non-trivial definable closure}\label{section:TDC}

The main result in \cite{AFP} actually holds without ultrahomogeneity under the weaker assumption of trivial (group-theoretic) definable closure:
\begin{definition}[Definable closure]\label{defn:TDC}
Let $\mathcal{L}$ be a signature and $\mathfrak{M}$ be an $\mathcal{L}$-structure.
For any $\vec x\in\Nb$, the {\em (group-theoretic) definable closure of $\vec x$ in $\mathfrak{M}$} is defined as
\[\dcl_{\mathfrak{M}}(\vec x):=\{b\in\Nb\mid g(b)=b\text{ for all automorphisms }g:\mathfrak{M}\to\mathfrak{M}\text{ that fix }\vec x\}.\]
We say that $\mathfrak{M}$ has {\em trivial definable closure} if $\dcl_{\mathfrak{M}}(\vec x)=\rng\vec x$ for all $\vec x\in\Nb$.
\end{definition}

Our construction in Theorem \ref{thm:main1}, however, requires ultrahomogeneity to construct the functions $\rho_{\mathfrak{S}}$.  
Example \ref{ex:TDC} shows that Theorem \ref{thm:main1} does not hold for general $\mathfrak{M}$ with trivial definable closure.
On the other hand, Austin \& Panchenko's \cite{AustinPanchenko2014} results for structures based in trees give representations when $\mathcal{L}'$ is a unary language where $\mathfrak{M}$ is ultrahomogeneous but fails to satisfy $n$-DAP, and in their stronger result, fails to even satisfy trivial definable closure.

These raise the following questions for future consideration.

\begin{question}
 Are there representations in the style of Theorem \ref{thm:nice1} and Theorem \ref{thm:main1} that hold when $\mathfrak{M}$ has trivial definable closure but is not ultrahomogeneous?
\end{question}

\begin{question}
Are there interesting classes of models with weaker properties than ultrahomogeneity and $n$-DAP for all $n$ with a stronger representation than that in Theorem \ref{thm:main1}?
\end{question}

%

\section*{Acknowledgements}
We thank Alex Kruckman for helpful discussions and comments on an earlier version.

\bibliography{refs}
\bibliographystyle{abbrv}

\end{document}